\documentclass[final,nomarks]{dmtcs-episciences}

\usepackage{amsmath,amsthm, amssymb, latexsym, url,amsfonts}

\usepackage{hyperref}

\usepackage{slashbox}

\usepackage{float}

\usepackage{graphicx}
\DeclareGraphicsExtensions{.eps, .pdf,.jpg,.png}
\DeclareGraphicsExtensions{.eps, .pdf,.jpg,.png}

\graphicspath{{Graphics/}}

\newtheorem{thm}{Theorem}[section]

\newtheorem{cor}[thm]{Corollary}
\newtheorem{lem}[thm]{Lemma}
\newtheorem{conj}[thm]{Conjecture}

\theoremstyle{definition}
\newtheorem{defn}[thm]{Definition}

\theoremstyle{remark}

\newtheorem*{ex}{Example}

\numberwithin{equation}{section}

\newcommand{\ZZ}{\mathbb{Z}}
\newcommand{\NN}{\mathbb{N}}
\newcommand{\RR}{\mathbb{R}}
\newcommand{\TT}{\mathbb{T}}

\newcommand{\fl}[1]{\lfloor #1\rfloor}
\newcommand{\Fl}[1]{\left\lfloor #1\right\rfloor}

\newcommand{\fp}[1]{\{ #1\}}

\newcommand{\ceil}[1]{\lceil #1\rceil}
\newcommand{\Ceil}[1]{\left\lceil#1\right\rceil}

\newcommand{\lgb}{\log_b}

\newcommand{\Sab}{{S_{a,b}}}
\newcommand{\Cab}{C_{a,b}}
\newcommand{\cab}{c_{a,b}}
\newcommand{\dab}{d_{a,b}}
\newcommand{\pab}{p_{a,b}}

\newcommand{\cbar}{\overline{c}_b}

\newcommand{\tildeb}{b_1}
\newcommand{\tilder}{r_1}

\newcommand{\bp}{b\,'}
\newcommand{\rp}{r\,'}
\newcommand{\dP}{d\,'}
\newcommand{\dpp}{d\,''}
\newcommand{\mpp}{m\,''}
\newcommand{\dzp}{d_0\kern-.8pt'}

\author{Xinwei He
\and A.J. Hildebrand
\and Yuchen Li
\and Yunyi Zhang
}

\title{Complexity of Leading Digit Sequences}

\affiliation{University of Illinois, USA}


\received{2018-4-9}
\revised{2019-11-5}
\accepted{2020-4-8}

\begin{document}

\publicationdetails{22}{2020}{1}{14}{4430}

\maketitle

\begin{abstract}
Let $\Sab$ denote the sequence of leading digits of $a^n$ in base $b$.
It is well known that if $a$ is not a rational power of $b$, 
then the sequence $\Sab$ satisfies Benford's Law; that is, 
digit $d$ occurs in $\Sab$ with  frequency
$\log_{b}(1+1/d)$, for $d=1,2,\dots,b-1$. 

In this paper, we investigate the \emph{complexity} of such sequences.
We focus mainly on the \emph{block complexity}, $\pab(n)$,
defined as the number of distinct blocks of length $n$ appearing in
$\Sab$.  In our main result we determine $\pab(n)$ for all
squarefree bases $b\ge 5$ and all rational numbers $a>0$
that are not integral powers of $b$.  In particular, we show that, for all
such pairs $(a,b)$, the complexity function $\pab(n)$ is  an \emph{affine}
function, i.e., of the form $\pab(n)=\cab n + \dab$ for all $n\ge1$, with 
coefficients $\cab\ge1$ and $\dab\ge0$, given explicitly  in terms of $a$
and $b$.  We also show that the requirement that $b$ be squarefree cannot
be dropped: If $b$ is not squarefree, then there exist integers $a$ with
$1<a<b$ for which $\pab(n)$ is not of the above form.

We use this result to obtain sharp upper and lower bounds for
$\pab(n)$ and to determine the asymptotic behavior of this function  as
$b\to\infty$ through squarefree values.  We also consider the question
which affine functions $p(n)=cn+d$ arise as the complexity function
$\pab(n)$ of some leading digit sequence $\Sab$.

We conclude with a discussion of other complexity measures for the
sequences $\Sab$ and some open problems. 
\end{abstract}

%%%%%%%%%%%%%%%%%%%%%%%%%%%%%%%%%%%%%%%%%%%%%%%%%%%%%%%%%%%%%%%%%%%%%%%%%
% body of paper
%%%%%%%%%%%%%%%%%%%%%%%%%%%%%%%%%%%%%%%%%%%%%%%%%%%%%%%%%%%%%%%%%%%%%%%%%

%%%%%%%%%%%%%%%%%%%%%%%%%%%%%%%%%%%%%%%%%%%%%%%%%%%%%%%%%%%%%%%%%%%%%%%%%
% intro/background 
%%%%%%%%%%%%%%%%%%%%%%%%%%%%%%%%%%%%%%%%%%%%%%%%%%%%%%%%%%%%%%%%%%%%%%%%%

\section{Introduction}
\label{sec:intro}

\subsection{Benford's Law}
The celebrated \emph{Benford's Law}, named after Frank Benford
\cite{benford1938}, states that leading digits in many 
data sets tend to follow the \emph{Benford distribution}, given by
\begin{equation}
\label{eq:benford} P(d) =\log_{10}\left(1+\frac1d\right),\quad
d=1,2,\dots,9. 
\end{equation}
Thus, in a data set following this distribution, approximately 
$\log_{10}2\approx 30.1\%$ of the numbers begin with digit $1$, 
approximately $\log_{10}(3/2)\approx 17.6\%$ begin with digit $2$, 
while only around $\log_{10}(10/9)\approx 4.6\%$ begin with digit $9$.

Benford's Law has been found to be a good match for a wide range of real
world data ranging from street addresses to populations of cities and
accounting data, and it has become an important tool in detecting tax and
accounting fraud.  Several books on the topic have appeared in recent years
(see, e.g., \cite{berger2015},
\cite{miller2015}, 
\cite{nigrini2012}),
and nearly one thousand articles have been published (see
\cite{benfordonline}).  

In recent decades, there has been a growing body of literature
investigating Benford's Law for mathematical sequences. Benford's Law has
been shown to hold (in the sense of asymptotic density) for 
large classes of sequences,  including exponentially growing sequences such as  
the powers of $2$ and the Fibonacci numbers, factorials, 
and the partition function; see, for example, 
Raimi \cite{raimi1976}, Diaconis \cite{diaconis1977}, 
Hill \cite{hill1995}, Anderson et al.  \cite{anderson2011},
and Mass\'e and Schneider \cite{masse2015}.

While the \emph{global} distribution 
and the \emph{global} fit to Benford's Law  
have been extensively investigated for large classes of arithmetic
sequences and are now well understood, the \emph{local} distribution
of such sequences remains to a large extent unexplored, and more mysterious.  
Recent work (see \cite{mersenne-benford} and \cite{local-benford}) revealed
that most (but not all) of the classes of
sequences that are known to satisfy Benford's Law
have poor \emph{local} distribution properties, in the sense that
$k$-tuples of leading digits of consecutive terms in the sequence do not
behave like $k$ independent Benford-distributed random variables. 
This is illustrated in Table \ref{table:leadingdigits}, which
shows the leading digits (in base $10$) 
of the first $50$ terms of the sequences $\{a^n\}$, 
$a=2,\dots,9$.  While the global distribution of digits in this
table is roughly as predicted by Benford's Law (for example, $30\%$ of the
$400$ digits in the table are $1$), the local distribution is completely
different: In some cases (e.g., for the sequence
$\{2^n\}$) the leading digits seem to
follow a near-periodic pattern, while in other cases (e.g., for the sequence 
$\{9^n\}$) they
show excessive repetition in leading digits.   In either case,
there is a strong dependence of leading digits of consecutive terms of the
sequence.

%%%%%%%%%%%%%%%%%%%%%%%%%%%%%%%%%%%%%%%%%%%%%%%%%%%%%%%%%%%%%%%%%%%%%%%%%
% table 
%%%%%%%%%%%%%%%%%%%%%%%%%%%%%%%%%%%%%%%%%%%%%%%%%%%%%%%%%%%%%%%%%%%%%%%%%

\begin{table}[H]

\begin{center}
\begin{tabular}{|c|c|}
\hline
Sequence & Leading digits of first $50$ terms (concatenated)
\\
\hline
\hline
$\{2^n\}$&
2481361251
2481361251
2481361251
2481361251
2481371251
\\ \hline

$\{3^n\}$&
3928272615
1514141313
1392827262
6151514141
3139282727
\\ \hline

$\{4^n\}$&
4162141621
4162141621
4172141721
4172141731
4173141731
\\ \hline

$\{5^n\}$&
5216317319
4216317319
4215217319
4215217319
4215217318
\\ \hline

$\{6^n\}$&
6321742116
3217421163
2174211632
1742116321
8421163218
\\ \hline

$\{7^n\}$&
7432118542
1196432117
5321196432
1175321196
4321175321
\\ \hline

$\{8^n\}$&
8654322111
8654322111
9754332111
9765432211
1865432211
\\ \hline

$\{9^n\}$&
9876554433
3222211111
1987765544
3332222111
1119877655

\\
\hline

\end{tabular}
\end{center}
\caption{Leading digits (in base $10$) of the first $50$
terms of the sequences $\{a^n\}$, $a=2,\dots,9$. 
}
\label{table:leadingdigits}
\end{table}

%%%%%%%%%%%%%%%%%%%%%%%%%%%%%%%%%%%%%%%%%%%%%%%%%%%%%%%%%%%%%%%%%%%%%%%%%
% end table 
%%%%%%%%%%%%%%%%%%%%%%%%%%%%%%%%%%%%%%%%%%%%%%%%%%%%%%%%%%%%%%%%%%%%%%%%%

Similar behavior can be found in more general sequences.  For example, in
\cite{local-benford} it is shown that sequences of the form $\{2^{p(n)}\}$,
where $p(n)$ is a polynomial, have excellent global, but poor local
distribution properties with respect to Benford's Law. On the other hand,
numerical data obtained in \cite{mersenne-benford} suggests that the leading
digits of the sequence $\{2^{p_n}\}$, where $p_n$ is the $n$-th prime,
satisfy Benford's Law on both the global and the local scale.

\subsection{Complexity of sequences}
In this paper we investigate leading digit sequences from the point of
view of \emph{complexity}.  The ``complexity'' of a sequence
$S=\{a_n\}$
over a finite set of symbols (for example, the digits $1,2,\dots,9$) can be
measured in a variety of ways; see the surveys of Allouche
\cite{allouche2012}, Ferenczi \cite{ferenczi1999}, and Kamae
\cite{kamae2012} for an overview of different complexity measures.  Here we
will use as our primary complexity measure the \emph{block complexity}%
\footnote{Equivalent terms for ``block complexity'' are 
\emph{subword complexity} and \emph{factor complexity}, 
with an infinite sequence being  considered an infinite \emph{word} over a
given \emph{alphabet}.  The terminology we are using here---block
complexity---is the one found in the \emph{mathematical} 
literature on the subject, e.g., the
surveys by Allouche \cite{allouche2012} and Ferenczi \cite{ferenczi1999}.}
defined as the function $p(n)=p_S(n)$ that counts the number of distinct
``blocks'' of length $n$ (i.e., $n$-tuples of consecutive terms) occurring
in a sequence $S$. 

The block complexity function $p_S(n)$ is the most commonly used complexity
measure for arithmetical sequences $S$ and has been extensively studied.
It is easy to see that the function $p_S(n)$ is bounded if and only if 
the sequence $S$ is eventually periodic.  On the other hand, for random
sequences, the block complexity function $p_S(n)$ grows at an exponential rate;
more precisely, it satisfies  $p_S(n)=k^n$, where $k$ is the number of distinct
symbols in the sequence.  In between these two extremes there is a rich
spectrum of sequences with intermediate levels of complexity and corresponding
rates of growth of $p_S(n)$.  We refer to the papers cited above---in
particular, Ferenczi \cite{ferenczi1999}---for further details, examples, and
references.

\subsection{The leading digit sequences $\Sab$}
Our main focus in this paper will be on leading digit sequences 
for geometrically growing sequences such as those  shown in Table
\ref{table:leadingdigits}. 
More precisely, given an integer $b\ge 3$ and a real number $a>0$, 
we consider the sequence $\Sab$
of leading digits of $a^n$ in base $b$; that is,  $\Sab$ is defined as 
\begin{equation}
\label{eq:def-Sab}
\Sab=\{D_b(a^n)\}_{n=1}^\infty,
\end{equation}
where $D_b(x)$ denotes the leading digit of $x$ in base $b$, defined by
\begin{equation}
\label{eq:Db(x)}
D_b(x)=d\Longleftrightarrow d\cdot b^k\le x< (d+1)b^k\quad
\text{for some $k\in\ZZ$} \quad (d=1,2,\dots,b-1).
\end{equation}
We denote by $\pab(n)$ the associated (block) complexity function,
i.e., the number of distinct blocks of length $n$ occurring in the
sequence $\Sab$.  More formally, $\pab(n)$ is given by 
\begin{equation}
\label{eq:def-pab}
\pab(n)=p_{\Sab}(n)
=\#\{(D_b(a^{m}),\dots,D_b(a^{m+n-1})): m=1,2,\dots\}.
\end{equation}

The data in Table \ref{table:leadingdigits} suggests that the sequences
$\Sab$, while not being periodic,  
have low complexity.  More extended computations confirm this: 
Figure \ref{fig:complexityBase10} shows the behavior of the ``empirical'' 
complexity 
functions $p_{a,10}(n)$ for selected values of $a$ and $n\le 100$, based on
the first $100,000$ terms of the sequence.\footnote{We use the term
``empirical'' here to emphasize the fact that the data were obtained 
by counting the number of distinct blocks of
length $n$ observed in a finite (though very large) initial segment of the
sequence and thus are not necessarily equal to the actual complexity
function.  However, the theoretical results we will prove here confirm the
data presented in Figure \ref{fig:complexityBase10}.}

%%%%%%%%%%%%%%%%%%%%%%%%%%%%%%%
% graphs of complexity functions 
%%%%%%%%%%%%%%%%%%%%%%%%%%%%%%%

\begin{figure}[H]
\begin{center}
\includegraphics[width=0.48\textwidth]{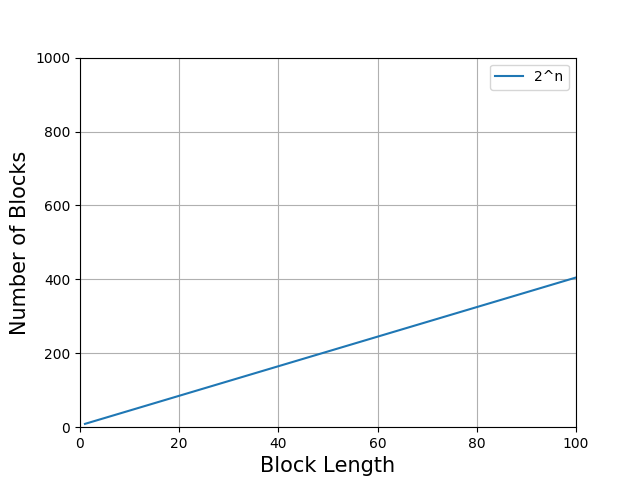}
\hspace{1em}
\includegraphics[width=0.48\textwidth]{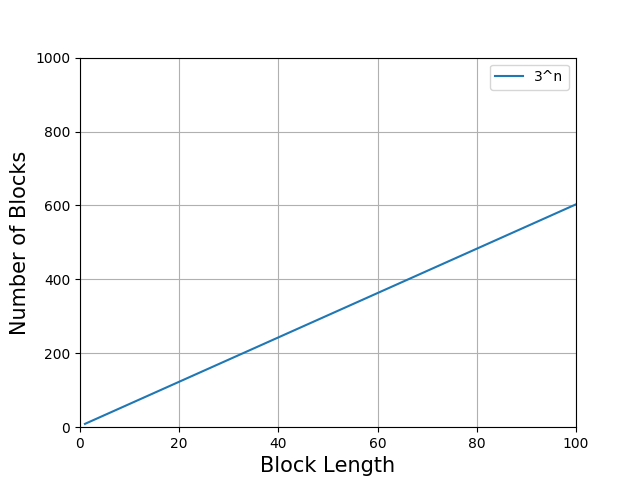}
\\
\includegraphics[width=0.48\textwidth]{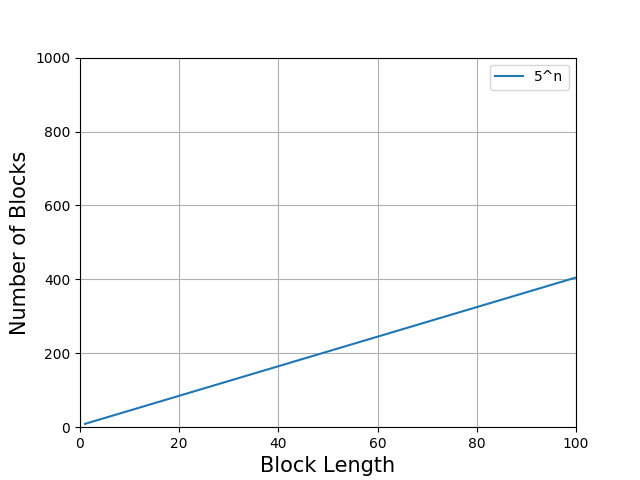}
\hspace{1em}
\includegraphics[width=0.48\textwidth]{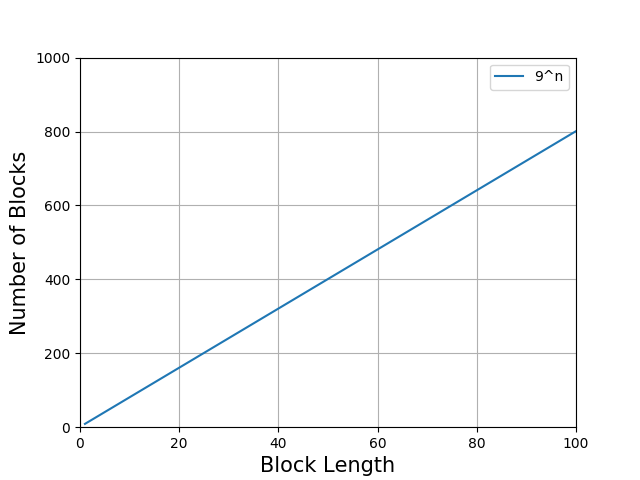}
\end{center}
\caption{Empirical complexity functions for the leading digit sequences of 
$\{2^n\},\{3^n\},\{5^n\},\{9^n\}$  in base $10$, 
based on the first $100,000$ terms of these sequences.
}
\label{fig:complexityBase10}
\end{figure}

The figure suggests that the functions $p_{a,10}(n)$ 
grow at a linear rate, with slopes depending on the value of $a$, though
the precise nature of this dependence is unclear.
Motivated by questions such as these, 
we seek to develop a complete understanding of
the complexity of the sequences $\Sab$.

\subsection{Coding sequences of rotations} 
Given real numbers $\alpha>0$ and $x$ and a partition of the unit interval
$0=\beta_0<\beta_1<\cdots<\beta_p=1$, the associated ``coding sequence''
$S=\{s_n\}$ is a sequence on $\{1,2,\dots,p\}$ defined by letting $s_n=k$ if and
only if $\{n\alpha + x \}\in [\beta_{k-1},\beta_{k})$ 
(where $\{t\}=t-\fl{t}$ denotes denotes the fractional part of $t$).
Such sequences have
been extensively studied in the literature; see, e.g., Alessandri and Berth\'e
\cite{alessandri-berthe}, and Berstel and Vuillon \cite{berstel2002}.  In
particular, it is known (see, e.g., \cite[Theorem 10]{alessandri-berthe})  that
the complexity function of a coding sequence with \emph{irrational} rotation
$\alpha$ is \emph{ultimately} affine, i.e., is of the form $p(n)=cn+d$ for
\emph{sufficiently large} $n$ (though in general not for \emph{all} $n\ge1$).

The leading digit sequences $\Sab$ defined above can be viewed as a special
type of coding sequence.  To see this, note that 
(cf. Lemma \ref{lem:leading-digit-criteria} below) $a^n$ has leading digit $d$ in
base $b$ if and only if $\{n\log_b a\}\in [\log_b d,\log_b(d+1))$, for
$d=1,2,\dots,b-1$.  Thus,
$\Sab$ is the coding sequence associated with the numbers $\alpha=\log_ba$ and
$x=0$, and the partition $0=\log_b1<\log_b 2<\dots < \log_b {(b-1)}< \log_b
b=1$.  It follows from the general result mentioned above that, if 
$\log_ba$ is irrational, then the complexity function $\pab(n)$ of $\Sab$ is
\emph{ultimately} affine, i.e., of the form $cn+d$ for all sufficiently large $n$.

In this paper we will show that, when $a$ is \emph{rational}, then, under some
mild additional assumptions, the complexity
function $\pab(n)$ is affine in the full sense, i.e., of the form $cn+d$ for
\emph{all} $n\ge 1$.

\subsection{Summary of results and outline of paper}
In Section \ref{sec:main} we state our main result, Theorem \ref{thm:main},
which completely determines the complexity function $\pab(n)$ of 
the leading digit sequence $\Sab$,
for any squarefree base $b\ge5$ and any positive 
rational number $a$ that is not an integral power of $b$.
We show that, under these assumptions, $\pab(n)$ is an affine function, i.e.,
satisfies \begin{equation}
\label{eq:pab-form}
\pab(n)=\cab n+\dab, \quad n=1,2,3,\dots,
\end{equation}
we give explicit formulas for the coefficients $\cab$ and  $\dab$ in
\eqref{eq:pab-form}, and we derive several corollaries from this result.

To complement Theorem \ref{thm:main},
we show in Theorem \ref{thm:main2} and Corollary \ref{cor:main2} that
the requirement that $b$ be squarefree cannot be dropped: 
For any non-squarefree integer  $b\ge5$ there exists an integer $a$ with $1<a<b$ 
such that the complexity function $\pab(n)$ is not of the form
\eqref{eq:pab-form} with $\cab\ge1$. 

In Section \ref{sec:proof}
we prove  Theorems \ref{thm:main} and
\ref{thm:main2}.  Our approach uses results and techniques from the theory
of dynamical systems generated by irrational ``shifts'' on the torus
$\TT$, along with some number-theoretic arguments.  

In Section \ref{sec:extreme} 
we consider extreme values of the complexity function
$\pab(n)$.  We show that, under the above assumptions 
on $a$ and $b$, the complexity function $\pab(n)$ satisfies
\[
\Fl{\frac{b-1}{2}}n 
+\Ceil{\frac{b-1}{2}} 
\le \pab(n)\le (b-1)n,\quad n=1,2,\dots,
\]
and that the upper and lower bounds are both sharp.

In Section \ref{sec:asymp} we determine the asymptotic behavior of
the ``slope'' $\cab$ in \eqref{eq:pab-form}
as $b\to\infty$ while $a$ is fixed. In particular,
we show that if $a$ is an integer $\ge 2$,
then the slope $\cab$ satisfies 
\[
\cab\sim \left(1-\frac1a\right)b
\]
as $b\to \infty$ through squarefree values.

In Section \ref{sec:values} we consider the question which complexity
functions $p(n)$ can be realized as the complexity function $\pab(n)$ of
a leading digit sequence $\Sab$ of the above type.
By \eqref{eq:pab-form} such a 
complexity function is necessarily affine. However,
not all affine functions $cn+d$ arise in this manner, and the question
of which pairs $(c,d)$ of coefficients correspond to leading digit
complexity functions leads to some interesting number-theoretic problems.

In Section \ref{sec:cyclomatic} we consider another complexity measure,
the  ``cyclomatic'' complexity, which has been originally developed as a
measure for the complexity of a graph and was adapted  to the context of
leading digit sequences by Iyengar et al. \cite{iyengar1983} and Kak
\cite{kak1983}.  We will determine the cyclomatic complexity for sequences of
the form $\Sab$.

In the final section, Section \ref{sec:concluding}, we discuss some related
work and present some open problems.

%%%%%%%%%%%%%%%%%%%%%%%%%%%%%%%%%%%%%%%%%%%%%%%%%%%%%%%%%%%%%%%%%%%%%%%%%
% Section 2:  Main Result  
%%%%%%%%%%%%%%%%%%%%%%%%%%%%%%%%%%%%%%%%%%%%%%%%%%%%%%%%%%%%%%%%%%%%%%%%%

\section{The complexity of $\Sab$: Main results}
\label{sec:main}

\subsection{Notations  and conventions}
We let $(n,m)$ denote the greatest common divisor of two integers $n$
and $m$.  We denote by $\fl{x}$ and $\ceil{x}$ the floor and ceiling
functions, defined as the largest integer $\le x$, resp. the smallest
integer $\ge x$. We let $\fp{x}=x-\fl{x}$ denote the fractional part 
of $x$.

Throughout this paper we assume that $b$ is an integer $\ge 3$ and
$a$ is a positive real number.  For our main results we will 
restrict $b$ and $a$ further as follows:

\begin{defn}[Admissible pairs]
\label{def:admissible}
A pair\footnote{The tuple notation, $(a,b)$, used in this definition
is also the notation for the greatest common
divisor. However, this will not cause any confusion as the 
meaning will always be clear from the context.}
$(a,b)$ is called \emph{admissible} if 
\begin{itemize}
\item[(i)] $b$ is a \emph{squarefree} integer $\ge 5$; and
\item[(ii)] $a$ is a positive \emph{rational} number that 
is not an integral power of $b$.
\end{itemize}
\end{defn}

Given an admissible pair $(a,b)$, we can represent the (rational) number 
$a$ uniquely in the form 
\begin{equation}
\label{eq:a=bkrs}
a=\frac{r}{s}\,b^k,
\quad k\in\ZZ,\quad r,s\in\NN, \quad (r,s)=1,
\quad 1<\frac{r}{s}<b.
\end{equation}
In particular, if $1<a<b$, then the integer $k$ in \eqref{eq:a=bkrs} is
$0$, so the representation \eqref{eq:a=bkrs} reduces to $a=r/s$, the standard
representation of $a$ as a reduced rational number.

\subsection{Main result}
We are now ready to state our main result, which completely describes the
complexity function of $\Sab$, for any admissible pair $(a,b)$.
\begin{thm}[Complexity of $\Sab$: Main Result]
\label{thm:main}
Let $(a,b)$ be an admissible pair and let $r$ and $s$ be defined by 
\eqref{eq:a=bkrs}.
Then the complexity function $\pab(n)$ of $\Sab$  satisfies
\begin{equation}
\label{eq:pab-rational}
\pab(n)=\cab n + \dab,\quad n=1,2,\dots,
\end{equation}
where
\begin{align}
\label{eq:cab}
\cab&=b-1-\Fl{\frac{b-1}{r}} -\Fl{\frac{(b,r)-1}{s}},
\\
\label{eq:dab}
\dab&=b-1-\cab= \Fl{\frac{b-1}{r}} +\Fl{\frac{(b,r)-1}{s}}.
\end{align}
\end{thm}

In particular, this result shows that $\pab(n)$ is an \emph{affine}
function for $n\ge1$, with $\cab$ and $\dab$ representing,
respectively, the slope and intercept of this function.
Note that, by \eqref{eq:dab},  $\cab$ and $\dab$ are related by the constraint
$\cab+\dab=b-1$.  Thus, the complexity function
$\pab(n)$ is completely determined by either of the quantities $\cab$ and
$\dab$ and the base $b$.

Table \ref{table:main-theorem-examples} gives a numerical illustration
of the formulas of Theorem \ref{thm:main}, showing
the complexity functions for the leading digit sequences $\Sab$ 
for $a=2,3,\dots,9$ and selected squarefree bases.
In particular, the results for $b=10$
confirm the empirical observations made in Figure
\ref{fig:complexityBase10}.

%%%%%%%%%%%%%%%%%%%%%%%%%%%%%%%
% table illustrating Theorem 1 
%%%%%%%%%%%%%%%%%%%%%%%%%%%%%%%

\begin{table}[H]
\begin{center}
\begin{tabular}{|c||c|c|c|c|c|c|c|c|}
\hline
Sequence $\{a^n\}$ &
$\{2^n\}$ &
$\{3^n\}$ &
$\{4^n\}$ &
$\{5^n\}$ &
$\{6^n\}$ &
$\{7^n\}$ &
$\{8^n\}$ &
$\{9^n\}$ 
\\
\hline\hline
$b=5$ &
$ 2n+2 $ & $ 3n+1$ & $ 3n+1$ & &$4n$  & $ 4n$ & $ 4n$ & $ 4n$

\\ \hline
$b=6$ &
$ 2n+3 $ & $ 2n+3 $ & $ 3n+2 $ & $ 4n +1 $ &   & $5n$ & $ 4n+1$ & $3n+2$
\\ \hline
$b=7$ &
$ 3n+3 $ & $ 4n+2 $ & $ 5n+1$ & $ 5n+1$ & $ 5n+1$ &  & $ 6n$  & $ 6n$
\\ \hline
$b=10$ &
$ 4n+5 $ & $ 6n+3 $ & $ 6n+3 $ & $ 4n+5 $ & $$
$ 7n+2 $ & $ 8n+1$ & $ 7n+2 $ & $ 8n+1$
\\
\hline
\end{tabular}
\vfill
\end{center}
\caption{The complexity functions $\pab(n)$ for selected values of $a$ and
$b$, computed using the formulas of Theorem \ref{thm:main}. Blank
entries correspond to pairs $(a,b)$ that are not admissible, i.e., cases 
where $a$ is an integral power of $b$.}
\label{table:main-theorem-examples}
\end{table}

It is natural to ask to what extent the 
restrictions imposed by the admissibility requirement can be relaxed.
The following remarks address this question:

\bigskip
\begin{itemize}
\item[(1)] The requirement that $a$ is not an integral power of $b$ 
serves to exclude trivial situations such as the sequence $\{10^n\}$ in
base $10$.  Indeed, it is not hard to see that whenever $a$ is a
\emph{rational} power of $b$, the sequence $\Sab$ 
is periodic, and hence has a bounded complexity
function.
We remark that, under the additional assumptions (which are part of the
admissibility condition) that $b$ is squarefree and $a$ is rational, the
two conditions ``$a$ is not an \emph{integral} power of $b$''  and ``$a$ is
not a \emph{rational} power of $b$'' are equivalent (cf. the proof of
Corollary \ref{cor:dirichlet} below).

\item[(2)] We have stated our result only for \emph{rational} values of $a$ as
this is the most interesting, and most challenging, case.  The  
result could be extended to irrational values of $a$, but complications
arise in certain special cases, such as the sequence $\{(\sqrt{2})^n\}$. 
One can show that for all but countably many irrational numbers $a$ one has
\begin{equation*}
\label{eq:pab-irrational}
\pab(n)=(b-1)n,\quad n=1,2,\dots.
\end{equation*}
In particular, $\pab(n)$ is of this form whenever
$a$ is a \emph{transcendental} number and $b$ an arbitrary integer $\ge
4$, not necessarily squarefree.
  
\item[(3)] The purpose of the restriction $b\ge 5$  is to avoid technical
complications that arise in the case $b=3$ and which would require a
separate treatment of this case. (These complications are due to the fact
that, when $b=3$, the interval $[\log_b1,\log_b2)$ has length $>1/2$,
whereas for $b\ge 4$ all intervals $[\lgb d, \lgb(d+1))$, $d=1,\dots,b-1$, 
have length $\le1/2$; cf. the footnote at the end of the proof of Lemma 
\ref{lem:p(k)-Lk}.)

\item[(4)] The most significant restriction in Theorem \ref{thm:main}
is the requirement that the base $b$ be squarefree. 
The results below show that this restriction is, in a sense, best-possible.
The restriction could, however,
be replaced by other restrictions involving both $a$ and $b$. For example,
we have $\pab(n)=(b-1)n$ for $n\ge1$ whenever $a$ and $b$ are positive
integers satisfying $a>b\ge4$ and $(a,b)=1$.
\end{itemize}

\begin{thm}[Complexity of $\Sab$: A counter-example]
\label{thm:main2}
Let $b\ge 5$ and $a\ge2$ be integers such that 
$a^2$ is a divisor of $b$, and $b$ is not a rational power of $a$. Then 
\begin{equation}
\label{eq:main2}
\pab(3)-\pab(2)\not=\pab(2)-\pab(1).
\end{equation}
In particular, under the above assumptions on $a$ and $b$, 
the complexity function $\pab(n)$ is not affine for $n\ge1$.
\end{thm}

\begin{cor}[Failure of Theorem \ref{thm:main} for non-squarefree bases]
\label{cor:main2}
Given any non-squarefree integer $b\ge5$, there exists an integer $a$ with
$1<a<b$ such that $\pab(n)$ is not of the form
\begin{equation}
\label{eq:main2cor}
\pab(n)=cn+d,\quad n=1,2,\dots
\end{equation}
for some integers $c\ge 1$ and $d$.
\end{cor}

\begin{proof}
Given a non-squarefree integer $b\ge5$, let $q$ be a prime such
that $q^2$ divides $b$, and take $a=q$. 
If $b$ is not a power of $q$, then Theorem \ref{thm:main2} yields the
desired conclusion.
If $b$ is a power of $q$, say $b=q^k$ with $k\ge2$, then
the sequence $\Sab$ of leading digits of $a^n(=q^n)$ in base
$b=q^k$ is the periodic sequence
$q,q^2,\dots,q^{k-1},1,q,q^2,\dots,q^{k-1},1,q,q^2,\dots$, and hence
has bounded complexity function $\pab(n)$. In particular,
\eqref{eq:main2cor} cannot hold with a \emph{positive} coefficient $c$.
\end{proof}

We remark that, while for non-squarefree bases $b\ge 5$ and values of $a$
that are not rational powers of $b$, the complexity
function $\pab(n)$ in general is not affine \emph{for all $n\ge 1$}, 
the general results about 
codings of irrational rotations mentioned above 
(e.g., \cite[Theorem 10] {alessandri-berthe}) 
imply that $\pab(n)$ is \emph{ultimately affine}, 
i.e., is of the form $\pab(n)=cn+d$ for $n\ge n_0$, for suitable
integers $n_0$, $c$ and $d$.  However, determining explicit values of 
the coefficients $c$ and $d$ and thus obtaining a result analogous to Theorem
\ref{thm:main} for non-squarefree values of $b$, seems to be a highly
nontrivial task.  (In particular, the formulas \eqref{eq:cab} and
\eqref{eq:dab} are, in general, not valid when $b$ is not squarefree.)

\subsection{Corollaries and special cases}

\begin{cor}[Special Case: Integer Values $a$]
\label{cor:integer-a}
Let $b\ge 5$ be squarefree and let  
$a$ be a positive integer that is not an integral power of $b$. Then 
we have:

\begin{itemize}
\item[(i)] If $1<a<b$, then
\begin{equation}
\label{eq:pab-a-integer1}
\pab(n)=
\left(b-\Fl{\frac{b-1}{a}}-(a,b)\right)
n + \Fl{\frac{b-1}{a}}+(a,b)-1,
\quad n=1,2,\dots
\end{equation}

\item[(ii)] If $a>b$ and $(a,b)=1$, then 
\begin{equation}
\label{eq:pab-integer2}
\pab(n)=(b-1)n, \quad n=1,2,\dots.
\end{equation}
\end{itemize}
\end{cor}

\begin{proof}
The assumptions on $a$ and $b$ ensure that the pair $(a,b)$ is admissible, 
so we can apply the formulas of Theorem \ref{thm:main}.
Since $\pab(n)=\cab n + \dab$ and $\dab=b-1-\cab$ (see \eqref{eq:dab}),
it suffices to show that
\begin{equation}
\label{eq:dab-integer}
\dab=\begin{cases}
\Fl{\frac{b-1}{a}}+(a,b)-1
&\text{if $1<a<b$,}
\\
0
&\text{if $a>b$ and $(a,b)=1$.}
\end{cases}
\end{equation}

Suppose first that $a$ is an integer satisfying $1<a<b$.
Then in \eqref{eq:a=bkrs} we have $r=a$ and $s=1$.
Hence the last term in 
\eqref{eq:dab} reduces to $(b,r)-1=(b,a)-1$, and 
\eqref{eq:dab-integer} follows.

Now suppose that $a$ is an integer satisfying $a>b$ and $(a,b)=1$. Then 
in the representation \eqref{eq:a=bkrs} we have $r=a$ and $s=b^k$, where 
$k$ is such that $b^k<a<b^{k+1}$.
Therefore
$(b,r)=(b,a)=1$, and hence
$\fl{((b,r)-1)/s}=0$.
Moreover, the assumption $a>b$ implies $\fl{(b-1)/r}=\fl{(b-1)/a}=0$.
Hence both terms on the right of \eqref{eq:dab} are $0$, and we obtain 
$\dab=0$, as claimed.
\end{proof}

\begin{ex} In base $b=10$, formula \eqref{eq:pab-a-integer1} gives 
$c_{a,10}=10-\fl{9/a}-(a,10)$ as the slope of the complexity function
of $S_{a,10}$,
Substituting $a=2,3,\dots,9$ in this formula yields the 
slopes $4,6,6,4,7,8,7,8$, respectively.
By \eqref{eq:dab}, the corresponding intercepts are given by 
$d_{a,10}=9-c_{a,10}$, so the associated complexity functions are 
$4n+5,6n+3,6n+3,4n+5,7n+2,8n+1,7n+2,8n+1$, respectively.  
These are the functions shown in the last row of Table
\ref{table:main-theorem-examples}.
\end{ex}

\begin{cor}[Special Case: $a=2$]
\label{cor:a=2}
Let $b$ be a squarefree integer $\ge 5$.
Then the complexity
function of the leading digit sequence of $2^n$  in base $b$ is given by 
\begin{equation}
\label{eq:pab-a=2}
p_{2,b}(n)=\Fl{\frac{b-1}{2}}n + \Ceil{\frac{b-1}{2}}, \quad n=1,2,\dots
\end{equation}
\end{cor}

\begin{proof}
The requirement that $b$ is squarefree and $\ge5$ 
ensures that $(2,b)$ is
admissible.  We can therefore apply 
Corollary \ref{cor:integer-a} to get 
\[
c_{2,b}=b-\Fl{\frac{b-1}{2}} -(2,b).
\]
If $b$ is even, this reduces to 
\[
c_{2,b}=b-\frac{b-2}{2} -2=\frac{b-2}{2}=\Fl{\frac{b-1}{2}},
\]
while for $b$ odd we get
\[
c_{2,b}=b-\frac{b-1}{2} -1=\frac{b-1}{2}=\Fl{\frac{b-1}{2}},
\]
so in either case we have 
\[
c_{2,b}=\Fl{\frac{b-1}{2}}.
\]
By the formulas \eqref{eq:dab} and \eqref{eq:pab-rational}, it follows
that 
\[
d_{2,b}=b-1-c_{2,b}=b-1-\Fl{\frac{b-1}{2}}=\Ceil{\frac{b-1}{2}}
\]
and 
\[
p_{2,b}(n)=c_{2,b} n + d_{2,b} = \Fl{\frac{b-1}{2}}n +
\Ceil{\frac{b-1}{2}}, 
\]
as claimed.
\end{proof}

\begin{cor}[Symmetry Property]
\label{cor:symmetry}
Let $(a,b)$ be an admissible pair. Then 
$(b/a,b)$ is admissible and the sequences $\Sab$ and
$S_{b/a,b}$ have the same complexity function, i.e., we have
\begin{equation}
\label{eq:symmetry}
\pab(n)=p_{b/a,b}(n), \quad n=1,2,\dots
\end{equation}
\end{cor}

\begin{ex}
The symmetry property can be used to explain some (but not all) 
of the coincidences of complexity functions shown in Table
\ref{table:main-theorem-examples}. 
For example, in base $10$ 
the leading digit sequences of $\{2^n\}$ and $\{5^n\}$ 
both have complexity function $4n+5$. 
In base $6$, the leading digit sequences of $\{2^n\}$ and $\{3^n\}$ 
both have complexity function $2n+3$. Since $5=10/2$ and $3=6/2$, 
these relations follow from the symmetry property.
\end{ex}

\begin{proof} 
First note that $a$ is an integral power of $b$ if and only if $b/a$ 
is an integral power of $b$. Thus $(a,b)$ is admissible if and only if
$(b/a,b)$ is admissible. This proves the first assertion of the theorem.

Let $(a,b)$ be an admissible pair.
To prove \eqref{eq:symmetry}, it
suffices to show that $\dab=d_{b/a,b}$, since this implies 
$\cab=c_{b/a,b}$ by the first identity in \eqref{eq:dab}, 
and hence $\pab(n)=\cab n+ \dab = p_{b/a,b}(n)$.

Replacing $a$ by $ab^{-k}$ with a suitable integer $k$ if
necessary, we may assume that $a$ lies in the range $1<a<b$.
Let $r/s$ be the representation of $a$ as a reduced rational number,
as given by  \eqref{eq:a=bkrs}.
Then \eqref{eq:dab} gives 
\begin{equation}
\label{eq:dab-symmetry1}
\dab=\Fl{\frac{b-1}{r}}+\Fl{\frac{(b,r)-1}{s}}.
\end{equation}

Now consider $a'=b/a$, and let $r'/s'$ be the representation of $a'$
as a reduced rational number.  Since $1<a<b$ we have $1<a'<b$, so formula
\eqref{eq:dab} applies again with $r$ and $s$ replaced by $r'$ and $s'$,
respectively, to give
\begin{equation}
\label{eq:dab-symmetry2}
d_{a',b}=\Fl{\frac{b-1}{r'}}+\Fl{\frac{(b,r')-1}{s'}}.
\end{equation}
To prove the result, it suffices to show that the expressions on the right
of \eqref{eq:dab-symmetry1} and \eqref{eq:dab-symmetry2} are equal.

Substituting $a=r/s$ into the definition of $a'$ gives 
\begin{equation}
\label{eq:symmetry-aprime}
a'=\frac{b}{a}
=\frac{bs}{r}=\frac{sb/(b,r)}{r/(b,r)} =\frac{s\tildeb }{\tilder }, 
\end{equation}
where
\begin{equation}
\label{eq:symmetry-btilde}
\tildeb =\frac{b}{(b,r)},\quad \tilder =\frac{r}{(b,r)}.
\end{equation}
Since $(\tildeb ,\tilder )=1$ and $(r,s)=1$, the numerator and
denominator in the fraction 
on the right of \eqref{eq:symmetry-aprime}
are coprime and hence must be equal to the quantities 
$r'$ and $s'$ in \eqref{eq:dab-symmetry2}; that is, we have 
\begin{equation}
\label{eq:symmetry-rprime}
r'=s\tildeb ,\quad s'=\tilder .
\end{equation}
It follows that
\begin{equation}
\label{eq:symmetry-gcd}
(b,r')=(b,s\tildeb )=\tildeb ((b,r),s)=\tildeb ,
\end{equation}
since $(r,s)=1$.
Substituting 
\eqref{eq:symmetry-rprime} and \eqref{eq:symmetry-gcd}
into \eqref{eq:dab-symmetry2}, we get 
\begin{equation}
\label{eq:dab-symmetry2a}
d_{a',b}=\Fl{\frac{b-1}{s\tildeb }}
+\Fl{\frac{\tildeb -1}{\tilder }}
=\Fl{\frac{(b,r)-1/\tildeb }{s}}
+\Fl{\frac{\tildeb -1}{\tilder }}.
\end{equation}
On the other hand, \eqref{eq:dab-symmetry1} can be written as 
\begin{equation}
\label{eq:dab-symmetry1a}
\dab=
\Fl{\frac{\tildeb -1/(b,r)}{\tilder }}
+\Fl{\frac{(b,r)-1}{s}}.
\end{equation}
Comparing \eqref{eq:dab-symmetry2a} and \eqref{eq:dab-symmetry1a}, 
we see that the equality of these expressions will follow
if we show that
\begin{equation*}
\Fl{\frac{(b,r)-1/\tildeb }{s}}
= \Fl{\frac{(b,r)-1}{s}}
\text{\ and\ }
\Fl{\frac{\tildeb -1/(b,r)}{\tilder }}
=
\Fl{\frac{\tildeb -1}{\tilder }}.
\end{equation*}
But this follows from the identity
\begin{equation*}
\Fl{\frac{h-x}{k}}=\Fl{\frac{h-1}{k}} \quad (0<x\le 1,\ h,k\in\NN),
\end{equation*}
which holds since
the open interval $((h-1)/k,h/k)$
does not contain an integer. 
\end{proof}

%%%%%%%%%%%%%%%%%%%%%%%%%%%%%%%
% section 3
%%%%%%%%%%%%%%%%%%%%%%%%%%%%%%%

\section{Proof of Theorems \protect\ref{thm:main} and
\protect\ref{thm:main2}}
\label{sec:proof}

We begin with two known results that we will need in the course of the
proof.  We recall that $D_b(x)$ denotes the leading digit of $x$ in base $b$,
as defined in \eqref{eq:Db(x)}, and that $\fp{t}$ denotes the fractional
part of $t$, defined as $\fp{t}=t-\fl{t}$.  The following lemma relates
$D_b(x)$ to the fractional part $\fp{\lgb x}$. This connection is well-known in
the literature on Benford's Law (see, e.g., \cite{benford1938} or
\cite{diaconis1977}).

\begin{lem}[Leading digit criterion]
\label{lem:leading-digit-criteria}
Let $x$ be a positive real number and $b$ an integer $\ge3$. Then
for any digit $d\in\{1,2,\dots,b-1\}$ we have
\begin{equation}
\label{eq:leading-digit-criteria}
D_b(x)=d\Longleftrightarrow \fp{\lgb x} \in [\lgb d, \lgb (d+1)).
\end{equation}
\end{lem}

\begin{proof}
By the definition of $D_b(x)$, we have 
\begin{align*}
D_b(x)=d&\Longleftrightarrow d\cdot b^k\le x< (d+1)b^k\quad
\text{for some $k\in\ZZ$}
\\
&\Longleftrightarrow \lgb d+ k\le \lgb x <  \lgb(d+1)+ k\quad
\text{for some $k\in\ZZ$}
\\
&\Longleftrightarrow \lgb d \le \fp{\lgb x} <  \lgb(d+1)
\\
&\Longleftrightarrow \fp{\lgb x} \in [\lgb d, \lgb (d+1)),
\end{align*}
where we used the fact that $0\le \lgb d<\lgb(d+1)\le 1$ for $1\le d\le
b-1$.
\end{proof}

The following result is well-known; see, e.g., \cite[Theorem
439]{hardy-wright}.

\begin{lem}
\label{lem:dirichlet}
Let $\alpha$ be an irrational number. Then the sequence $\{n \alpha\}$ is
dense in the interval $[0,1)$.
\end{lem}

\begin{cor}
\label{cor:dirichlet}
If $(a,b)$ is an admissible pair, then the sequence $\{n\lgb a\}$ is dense
in the interval $[0,1)$.
\end{cor}

\begin{proof}
Assume $(a,b)$ is an admissible pair.  By Lemma \ref{lem:dirichlet}
it suffices to show that, if $(a,b)$ is admissible, then $\lgb a$ is irrational,
i.e., $a$ is not a rational power of $b$.

We argue by contradiction. Suppose $(a,b)$ is admissible and $a$ is a rational
power of $b$, say $a=b^{p/q}$, where $p$ and $q$ are coprime positive integers.
The admissibility condition implies that $a$ is a positive rational number,
i.e., of the form $a=m/n$, where $m$ and $n$ are coprime positive integers, and
that $b$ is squarefree, i.e., of the form $b=p_1\dots p_k$, where the $p_i$ are
distinct primes. Substituting these representations into the relation
$a=b^{p/q}$, we obtain $m/n=(p_1\dots p_k)^{p/q}$, or equivalently
$m^q=n^qp_1^p\dots p_k^p$. 

Since $m$ and $n$ are coprime, this can only hold
if $n=1$. Hence we must have $m^q=p_1^p\dots p_k^p$. 
By the Fundamental Theorem of Arithmetic it follows that $m$ must of the form
$m=p_1^{\alpha_1}\dots p_k^{\alpha_k}$ with nonnegative integer 
exponents $\alpha_i$, and that $q\alpha_i = p$ for all $i$. 
This is only possible if $p/q=h$ is a positive integer and
$\alpha_i=h$ for all $i$, i.e., if $m=(p_1\dots p_k)^h=b^h$.
But then $a=m/n=m=b^h$ is an integral power of $b$, 
contradicting the admissibility condition.  This completes the proof.
\end{proof}

For the remainder of this section we fix an admissible pair $(a,b)$. 
Thus $b\ge 5$ is squarefree and $a$ is not an integral power of $b$.
We remark that the assumption that $b$ is squarefree will only be needed
in the latter part of the proof of Theorem \ref{thm:main} (beginning with
Lemma \ref{lem:Lk}); Lemmas \ref{lem:p(k)-Lk} and \ref{lem:L1-2} hold
without this assumption.

Dividing $a$ by  a power of $b$
if necessary, we may assume without loss of generality that $1<a<b$,
so that $a=r/s$, where $r$ and $s$ are as in Theorem \ref{thm:main} (see
\eqref{eq:a=bkrs}). The assumption $1<a<b$ then implies 
\begin{equation}
\label{eq:rs-bound}
1<\frac{r}{s}<b.
\end{equation}

We introduce the following notations: 
\begin{align}
\label{eq:def-a}
\alpha&= \lgb a = \lgb \frac{r}{s},
\\
\label{eq:def-D}
D&=\{1,2,\dots,b-1\},
\\
\label{eq:def-L}
L&=\lgb D = \{\lgb 1, \lgb 2,\dots,\lgb(b-1)\},
\\
\label{eq:def-Lk}
L_k&=\bigcup_{i=0}^{k-1}(L-i\alpha)=\{\lgb d -i\lgb a: d\in D,\
i=0,1,\dots,k-1\}.
\end{align}
We regard the sets $L_k$ as subsets of the one-dimensional torus
$\TT=\RR/\ZZ$ by identifying elements that differ by an integer.
The following key result relates the sets $L_k$ to the complexity functions
that we seek to evaluate.  More general results of this type are known  
in the context of codings of irrational rotations (see, e.g., \cite[Theorem
10]{alessandri-berthe}). For the sake of completeness, we provide a
self-contained proof here. 

\begin{lem}
\label{lem:p(k)-Lk}
Let $p(k)=\pab(k)$ be the complexity function of the leading digit
sequence $\Sab$. Then 
\begin{equation}
\label{eq:p(k)-Lk}
p(k)=|L_k|\quad (k=1,2,\dots),
\end{equation}
where $|L_k|$ denotes the cardinality of $L_k$.
\end{lem}

\begin{proof}
Recall that $p(k)$ denotes the number of blocks of length $k$ in the
sequence $\Sab$, i.e., the number of distinct 
tuples $(d_0,\dots,d_{k-1})$ of digits $d_i\in D$ such that, 
for some $n\in\NN$, 
\begin{equation}
\label{eq:p(k)-Lk1}
D_b(a^{n+i})=d_i,\quad i=0,1,\dots,k-1.
\end{equation}
Using Lemma \ref{lem:leading-digit-criteria} we see that
(recall that, by \eqref{eq:def-a}, $\alpha=\lgb a$)
\begin{align*}
D_b(a^{n+i})=d_i
&\Longleftrightarrow \fp{\lgb (a^{n+i})}\in [\lgb d_i, \lgb (d_i+1))
\\
&\Longleftrightarrow \fp{(n+i)\lgb a}\in [\lgb d_i, \lgb (d_i+1))
\\
&\Longleftrightarrow \fp{(n+i)\alpha}\in [\lgb d_i, \lgb (d_i+1))
\\
&\Longleftrightarrow n\alpha\in [\lgb d_i-i\alpha, \lgb
(d_i+1)-i\alpha)+m_i
\quad\text{for some $m_i\in\ZZ$.}
\end{align*}
It follows that \eqref{eq:p(k)-Lk1} holds if and only if 
\begin{equation}
\label{eq:p(k)-Lk2}
n\alpha\in \bigcap_{i=0}^{k-1}
\Bigl[\lgb d_i-i\alpha+m_i, \lgb (d_i+1)-i\alpha+m_i\Bigr)
\end{equation}
for some $m_0,\dots,m_{k-1}\in\ZZ$.
Interpreting both sides of \eqref{eq:p(k)-Lk2} as elements of 
$\TT=\RR/\ZZ$, we can rewrite this relation as  
\begin{equation}
\label{eq:p(k)-Lk3}
n\alpha\in \bigcap_{i=0}^{k-1}
\Bigl[\lgb d_i-i\alpha, \lgb (d_i+1)-i\alpha\Bigr)\quad
\text{in $\TT$.}
\end{equation}

Now, note that, by Corollary \ref{cor:dirichlet}, the sequence $\{n\alpha\}$ is
dense in the unit interval $[0,1)$.  Thus, if the interval on the right of 
\eqref{eq:p(k)-Lk3} is non-empty, it must contain an element of this sequence.
Hence, given any $k$-tuple 
$(d_0,\dots,d_{k-1})$ of digits in $D$ for which the interval  
on the right of \eqref{eq:p(k)-Lk3} is non-empty, there exists an
$n\in\NN$ such that \eqref{eq:p(k)-Lk1} holds for this tuple, i.e.,  
the tuple $(d_0,\dots,d_{k-1})$ occurs as a 
block of length $k$ in the sequence $\Sab$.
Conversely, if $(d_0,\dots,d_{k-1})$ is a block of length $k$ occurring
in $\Sab$, then there exists an $n$ such that 
relation \eqref{eq:p(k)-Lk3} holds, so the interval on the right of
\eqref{eq:p(k)-Lk3} must be non-empty. 

It follows%
\footnote{Our
assumption $b\ge5$ ensures that each of the intervals $[\lgb d, \lgb(d+1))$,
$d=1,2,\dots,b-1$, has length $\le1/2$. Hence the intersection of such
an interval with translates of other intervals of this
type is either empty or consists of a \emph{single} interval in $\TT$.
This is not necessarily true for intervals of length $>1/2$:
for example, the intersection of the
interval $[0,2/3]$ with its translate by $1/2$ consists of the two disjoint
intervals $[0,1/6]$ and $[1/2,2/3]$.}
that \emph{the number of 
blocks of length $k$ in the sequence $\Sab$
(and hence the value of the complexity function $p(k)$) 
is equal to the number of non-empty intervals in $\TT$ generated 
on the right of \eqref{eq:p(k)-Lk3} as each
$d_i$ runs through the digits in $D=\{1,2,\dots,b-1\}$.}
But these intervals are
exactly the intervals obtained by splitting up $\TT$ at the points
\begin{equation}
\label{eq:p(k)-Lk4}
\lgb d- i \alpha \in \TT,\quad d\in D,\quad i=0,1,\dots,k-1,
\end{equation}
so the number of such intervals is equal to the number of \emph{distinct} 
elements in \eqref{eq:p(k)-Lk4}. The latter elements form the
elements of the set $L_k$, so the desired number
is  $|L_k|$.
This completes the proof of 
Lemma \ref{lem:p(k)-Lk}.
\end{proof}

To complete the proof of Theorem \ref{thm:main}, it remains to evaluate
the numbers $|L_k|$.  As mentioned, we consider the sets
$L_k$ as subsets of $\TT$. Thus, in what follows relations involving the
elements of these
sets are to be interpreted as relations among elements in $\TT$, i.e., as
relations that hold modulo $1$.

\begin{lem}
\label{lem:L1-2}
We have 
\begin{align}
\label{eq:L1}
|L_1|&=b-1,
\\
\label{eq:L2}
|L_2|&=2(b-1)-\Fl{\frac{b-1}{r}}-\Fl{\frac{(b,r)-1}{s}}.
\end{align}
\end{lem}

\begin{proof} 
By definition, $L_1$ is the set $L=\{\lgb1,\dots,\lgb(b-1)\}$, which has
$b-1$ distinct elements in $\TT$. Thus, $|L_1|=|L|=b-1$, proving
\eqref{eq:L1}.

Now consider $L_2$. By definition, $L_2=L\cup (L-\alpha)$. Thus,
\begin{equation*}
\label{eq:L2-proof1}
|L_2|=|L|+|L-\alpha|-|L\cap (L-\alpha)| = 2(b-1)-|L\cap (L-\alpha)|.
\end{equation*}
To prove \eqref{eq:L2}, it therefore suffices to show 
\begin{equation}
\label{eq:L2var}
|L\cap(L-\alpha)|=\Fl{\frac{b-1}{r}}+\Fl{\frac{(b,r)-1}{s}}.
\end{equation}

For the proof of \eqref{eq:L2var}, consider 
an element $x\in L\cap (L-\alpha)$. Since $x\in L$, $x$ must be of the
form 
\begin{equation*}
x=\lgb d 
\end{equation*}
for some $d\in D$, and since $0\le \lgb 1 \le \lgb d\le
\lgb (b-1)<1$, the digit $d$ is uniquely determined by $x$.

Similarly, since $x\in L-\alpha$, $x$ must also be of the form 
\begin{equation*}
x=\lgb \dP -\lgb a + m
\end{equation*}
for some $\dP \in D$ and
$m\in \ZZ$. 
We therefore have
\begin{equation*}
\lgb d = \lgb \dP -\lgb a + m,
\end{equation*}
or equivalently,
\begin{equation*}
d= \dP a^{-1} b^m.
\end{equation*}
Since $a=r/s$, the latter relation can be written as
\begin{equation}
\label{eq:L2-proof2}
dr=\dP sb^m.
\end{equation}

We next show that the integer $m$ in \eqref{eq:L2-proof2}  (and hence
also $\dP $) is uniquely determined by $d$ (and hence by $x$), 
and that $m$ must be either
$0$ or $1$.
Since $1/(b-1)\le d/\dP \le b-1$ and $1<r/s<b$
(see \eqref{eq:rs-bound}),
\eqref{eq:L2-proof2}
implies
\begin{equation*}
b^m=(d/\dP )(r/s)< (b-1)b<b^2
\end{equation*}
and 
\begin{equation*}
b^m=(d/\dP )(r/s)> \frac{1}{b-1}> \frac1{b},
\end{equation*}
so we have $-1<m<2$ and hence either $m=0$ or $m=1$. 
Moreover, for each given element $d\in D$ 
at most one of these cases can occur. 
Indeed, rewriting \eqref{eq:L2-proof2} 
as $\dP =d(r/s)b^{-m}$, we see that the integer $m$ (if it exists) is
uniquely determined by the requirement that $1\le \dP \le b-1$. 

Therefore we have 
\begin{equation}
\label{eq:L2var2}
|L\cap (L-\alpha)|=N_0+ N_1,
\end{equation}
where $N_0$ (resp. $N_1$) is the number of elements $d\in D$ satisfying
\eqref{eq:L2-proof2} for some $\dP \in D$ with $m=0$ (resp.  $m=1$).
We will show that $N_0$ and $N_1$ are equal to the two terms 
on the right of \eqref{eq:L2var}.

Consider first the case $m=0$.  Then  \eqref{eq:L2-proof2} reduces to 
\begin{equation}
\label{eq:L2-proof2a}
dr=\dP s.
\end{equation}
Since $r$ and $s$ are
relatively prime, \eqref{eq:L2-proof2a}
can only hold if $d$ is a multiple
of $s$ and $\dP $ is the same multiple of $r$, i.e., if $d=d_0s$ and
$\dP =d_0r$ for some positive integer $d_0$.  Furthermore, since $\dP \le
b-1$, we have
\begin{equation}
\label{eq:L2-proof2ab}
d_0=\frac{\dP }{r}\le \frac{b-1}{r}.
\end{equation}

We claim that each positive integer $d_0$ satisfying
\eqref{eq:L2-proof2ab} yields a pair $(d,\dP )$
of digits in $D$ for which \eqref{eq:L2-proof2a} holds.
Indeed, setting $d=d_0s$ and $\dP =d_0r$, we have $dr=\dP s$, and the bound 
$d_0\le (b-1)/r$, along with our assumption $r/s>1$ 
(see \eqref{eq:rs-bound})
imply that both $d$ and $\dP $ are positive integers 
bounded by $\le b-1$ and hence are elements in $D$.
Thus, the number of elements in the set $L\cap (L-\alpha)$ arising 
from case $m=0$ is equal to the number of positive integers $d_0$
satisfying \eqref{eq:L2-proof2ab}, i.e., we have
\begin{equation}
\label{eq:L2-proof2aa}
N_0=\Fl{\frac{b-1}{r}}.
\end{equation}

Now suppose that $m=1$. Then \eqref{eq:L2-proof2} reduces to
\begin{equation}
\label{eq:L2-proof2b}
dr=\dP sb.
\end{equation}
Set
\[
\bp =\frac{b}{(b,r)}, \quad 
\rp =\frac{r}{(b,r)}.
\]
Dividing through by $(r,b)$, \eqref{eq:L2-proof2b} becomes 
$d\rp =\dP s\bp $. Since $\rp $ is coprime with both $s$ and $\bp $, 
this can only hold if 
\[
\dP =\dzp  \rp \quad\text{and}\quad d=\dzp s\bp 
\]
for some positive integer $\dzp $.
Since $d\le b-1$, the integer $\dzp$ must satisfy 
\begin{equation}
\label{eq:L2-proof2c}
\dzp \le \frac{b-1}{s\bp }.
\end{equation}
Conversely, every positive integer $\dzp $ satisfying \eqref{eq:L2-proof2c}
yields a pair $(d,\dP )$ of digits in $D$ satisfying \eqref{eq:L2-proof2b}.
Indeed, setting $d=\dzp s\bp $ and $\dP =\dzp \rp $,
we have 
\begin{equation*}
dr=d(r,b)\rp =\dzp s\bp (r,b)\rp =\dzp \rp sb=\dP sb,
\end{equation*}
so \eqref{eq:L2-proof2b} holds.
Moreover, $d$ and $\dP $ are both positive integers 
and the bound \eqref{eq:L2-proof2c} ensures that 
\[
d=\dzp s\bp  \le b-1
\]
and 
\[
\dP =\dzp \rp \le \frac{(b-1)}{s\bp }\rp =\frac{(b-1)(r/s)}{b}< b-1,
\]
where in the last step we used the bound $r/s<b$.
Thus the contribution of the case $m=1$ to the set $L\cap (L-\alpha)$ 
is equal to the number of positive integers $\dzp $ satisfying
\eqref{eq:L2-proof2c}, i.e., we have
\begin{equation}
\label{eq:L2-proof2bb}
N_1=\Fl{\frac{b-1}{s\bp }}
=\Fl{\frac{b}{s\bp }-\frac{1}{s\bp }}
=\Fl{\frac{(b,r)-1/\bp }{s}}
=\Fl{\frac{(b,r)-1}{s}}.
\end{equation}
Substituting \eqref{eq:L2-proof2aa}
and \eqref{eq:L2-proof2bb} into 
\eqref{eq:L2var2} 
yields the desired relation \eqref{eq:L2var}. 
\end{proof}

Up to this point, our argument did not make use of the assumption that $b$
be squarefree.  The following lemma, however, depends on this assumption in
a crucial manner.

\begin{lem}
\label{lem:Lk}
For any positive integer $k$ we have
\begin{equation}
\label{eq:Lk}
|L_{k+1}| - |L_k| = |L_2| - |L_1|.
\end{equation}
\end{lem}

\begin{proof}
By the definition of the sets $L_k$ we have
\begin{align*}
|L_{k+1}|
&=\left|\bigcup_{i=0}^k(L-i\alpha)\right|
=\left|\bigcup_{i=1}^k(L-i\alpha)\right| +|L|
-\left|L\cap\left(\bigcup_{i=1}^k(L-i\alpha)\right)\right|.
\end{align*}
Since
\begin{align*}
\left|\bigcup_{i=1}^k(L-i\alpha)\right|
&=\left|\bigcup_{i=1}^k(L-i\alpha)+\alpha\right|
=\left|\bigcup_{i=0}^{k-1}(L-i\alpha)\right|=|L_k|,
\end{align*}
it follows that 
\begin{equation}
\label{eq:Lk-proof1}
|L_{k+1}| - |L_k| = |L|  
-\left|L\cap\left(\bigcup_{i=1}^k(L-i\alpha)\right)\right|.
\end{equation}
Now note that for $k=1$ the right-hand side of 
\eqref{eq:Lk-proof1} reduces to $|L|-|L\cap (L-\alpha)|$, whereas the
left-hand side becomes $|L_2|-|L_1|$.  Thus, 
to prove the desired relation \eqref{eq:Lk}, it suffices to show that 
\begin{equation*}
\label{eq:Lk-proof2}
\left|L\cap\left(\bigcup_{i=1}^k(L-i\alpha)\right)\right|
=\left|L\cap(L-\alpha)\right|.
\end{equation*}
This in turn will follow if we can show that, for any positive integer
$i\ge 2$,  
\begin{equation}
\label{eq:Lk-proof3}
L\cap (L-i\alpha)\subset L \cap (L-\alpha).
\end{equation}

It remains to prove \eqref{eq:Lk-proof3}. Fix $i\ge 2$, and consider an
element $x\in L\cap (L-i\alpha)$. Since $x\in L$, $x$ must be of the form
\begin{equation}
\label{eq:Lk-proof4}
x=\lgb d
\end{equation}
for some $d\in D$. Since $x\in L- i\alpha$, $x$ must also be of the form
\begin{equation}
\label{eq:Lk-proof5}
x=\lgb \dP  - i\lgb a + m 
\end{equation}
for some $\dP \in D$ and $m\in\ZZ$.
We need to show that $x\in L-\alpha$, i.e., that
\begin{equation}
\label{eq:Lk-proof6}
x=\lgb \dpp  - \lgb a + \mpp  
\end{equation}
for some $\dpp \in D$ and $\mpp \in\ZZ$.

From \eqref{eq:Lk-proof4}
and \eqref{eq:Lk-proof5} we get 
\begin{equation*}
\lgb d = \lgb \dP  -i \lgb a + m
\end{equation*}
and hence 
\begin{equation*}
d=\dP a^{-i} b^m.
\end{equation*}
Using $a=r/s$, this can be written as 
\begin{equation}
\label{eq:Lk-proof7}
dr^i =\dP  b^ms^i. 
\end{equation}
On the other hand, by \eqref{eq:Lk-proof4} 
the desired relation \eqref{eq:Lk-proof6} is equivalent to 
\[
\lgb d=\lgb \dpp  - \lgb a + \mpp , 
\]
i.e., 
\begin{equation}
\label{eq:Lk-proof-goal}
dr=\dpp b^{\mpp }s,
\end{equation}
where $\dpp \in D$ and $\mpp \in \ZZ$.
Thus, we seek to show that if \eqref{eq:Lk-proof7} holds for some
$\dP \in D$ and $m\in\ZZ$, 
then \eqref{eq:Lk-proof-goal} holds for some 
$\dpp \in D$ and $\mpp \in\ZZ$. 

Let
\begin{equation}
\label{eq:Lk-proof8}
b_0=(b,r),\quad \bp =\frac{b}{b_0},\quad \rp =\frac{r}{b_0}.
\end{equation}
By our assumption that $b$ is squarefree we have $(\bp ,b_0)=1$ and 
since also $(\bp ,\rp )=1$, it follows that  
\begin{equation}
\label{eq:Lk-proof9}
(\bp ,r)=(\bp ,\rp b_0) = 1.
\end{equation}
Substituting $b=b_0\bp $ in \eqref{eq:Lk-proof7}, we obtain 
\begin{equation}
\label{eq:Lk-proof10}
\dP  b_0^m \bp\,^m s^i = dr^i.
\end{equation}
Now note that since $r/s>1$, \eqref{eq:Lk-proof7} 
implies $b^m>(d/\dP )\ge 1/(b-1)$, so the integer $m$ 
in \eqref{eq:Lk-proof10}
must satisfy $m\ge 0$.
We consider the cases $m=0$ and $m\ge 1$ separately.

Suppose first that $m=0$. Then \eqref{eq:Lk-proof10} reduces to 
\begin{equation}
\label{eq:Lk-proof10a}
\dP s^i=dr^i.
\end{equation}
Since $(s,r)=1$, it follows that 
$s^i$ divides $d$, i.e., we have
\begin{equation}
\label{eq:Lk-proof11}
d=d_0s^i
\end{equation}
for some positive integer $d_0$. Now let
\begin{equation*}
\label{eq:Lk-proof12}
\dpp =d(r/s).
\end{equation*}
Then $\dpp s=dr$, so \eqref{eq:Lk-proof-goal} holds with $\mpp =0$.
Moreover, by \eqref{eq:Lk-proof11} we have 
\begin{equation*}
\dpp =d(r/s)=d_0 s^i(r/s)=d_0rs^{i-1}, 
\end{equation*}
so $\dpp $ is a positive integer. In addition, $\dpp $ satisfies the upper bound
\begin{equation*}
\dpp = d(r/s)\le d(r/s)^i=\dP \le b-1,
\end{equation*}
where we have used the inequality $r/s>1$
along with the relation \eqref{eq:Lk-proof10a}.
Thus, $\dpp $ is an element in $D$ satisfying \eqref{eq:Lk-proof-goal} with
$\mpp =0$. Thus we have reached the desired conclusion in the case $m=0$. 

Now suppose that $m\ge1$ in \eqref{eq:Lk-proof10}. 
Then $\bp s$ divides the left-hand side of \eqref{eq:Lk-proof10}, and hence
also the right-hand side, i.e., $dr^i$.
Since $(s,r)=1$ and $(\bp ,r)=1$ (see \eqref{eq:Lk-proof9}), it follows
that $\bp s$ divides $d$. Thus, we have 
\begin{equation}
\label{eq:Lk-proof13}
d=\dzp \bp s
\end{equation}
for some positive integer $\dzp $.
Let
\begin{equation}
\label{eq:Lk-proof14}
\dpp =\dzp \rp .
\end{equation}
Then, by \eqref{eq:Lk-proof13},  
\begin{equation}
\label{eq:Lk-proof15}
\dpp bs=\dzp \rp bs=\dzp (r/b_0)(\bp b_0)s=(\dzp \bp s)r=dr,
\end{equation}
so
\eqref{eq:Lk-proof-goal} holds with $\mpp =1$.
Moreover, \eqref{eq:Lk-proof14} shows that 
$\dpp $ is a positive integer, and by \eqref{eq:Lk-proof15} and 
the bound $r/s<b$ (see \eqref{eq:rs-bound}),
$\dpp $ is bounded above by 
\begin{equation*}
\dpp \le \frac{dr}{bs} <d\le b-1.
\end{equation*}
Thus, $\dpp \in D$, and the proof of the lemma is complete.
\end{proof}

\begin{proof}[of Theorem \ref{thm:main}]
We need to show that
\begin{equation}
\label{eq:proof-main1}
p(k)=\cab k + \dab,\quad k=1,2,\dots,
\end{equation}
where 
\begin{align*}
\cab &= b-1 -
\Fl{\frac{b-1}{r}}-\Fl{\frac{(b,r)-1}{s}},
\\
\dab &= b-1-\cab= \Fl{\frac{b-1}{r}}+\Fl{\frac{(b,r)-1}{s}}.
\end{align*}

By Lemma \ref{lem:p(k)-Lk} we have $p(k)=|L_k|$ for any positive integer $k$,
so it suffices to show that \eqref{eq:proof-main1}
holds with $|L_k|$ in place of $p(k)$.

By the first part of Lemma \ref{lem:L1-2} we have 
\begin{equation}
|L_1|=b-1=\cab\cdot 1+\dab, 
\end{equation}
so \eqref{eq:proof-main1} holds for $k=1$. 
By the second part of Lemma
\ref{lem:L1-2} and Lemma \ref{lem:Lk} we have, for any $k\ge 2$, 
\begin{align*}
|L_k|&
=|L_1|+\sum_{i=1}^{k-1}\left(|L_{i+1}|-|L_i|\right)
\\
&=|L_1|+(k-1)\left(|L_2|-|L_1|\right)
\\
&=\cab+\dab+(k-1)\cab = \cab k + \dab,
\end{align*}
which, combined with the relation $|L_k|=p(k)$,
proves the desired formula \eqref{eq:proof-main1}. 
This completes the proof of Theorem \ref{thm:main}.
\end{proof}

\begin{proof}[of Theorem \ref{thm:main2}]
Assume that $b\ge 5$ and $a\ge 2$ are integers such that 
$a^2$ divides $b$ and $b$ is not a rational power of $a$.
We will show that in this case 
the complexity function $p(n)=p_{a,b}(n)$ satisfies 
$p(3)-p(2)\not=p(2)-p(1)$, and hence is not affine.

Since $a^2$ divides $b$ and $b$ is not a rational power of $a$, we have
\begin{equation}
\label{eq:main2aa}
b=a^2b_1
\end{equation}
for some integer $b_1\ge2$. 

By Lemmas \ref{lem:p(k)-Lk} and \ref{lem:L1-2} (which, as noted above, 
do not depend on the assumption that $b$ is squarefree) it suffices to show 
\begin{equation}
\label{eq:main2a}
|L_3|-|L_2|\not = |L_2|-|L_1|.
\end{equation}
Now, as in the proof of Lemma \ref{lem:Lk} (see \eqref{eq:Lk-proof3}) 
we see that  \eqref{eq:main2a} is equivalent to 
\begin{equation}
\label{eq:main2b}
L\cap (L-2\alpha)\not\subset L\cap (L-\alpha).
\end{equation}
Thus, it suffices to construct an element $x\in L$ such that 
$x\in L-2\alpha$, but $x\not\in L-\alpha$. (Recall that  the sets $L$ are
to be understood as subsets in $\TT=\RR/\ZZ$, so all relations are
to be interpreted as relations that hold modulo $1$.).

We take 
\begin{equation}
\label{eq:main2c}
x=\lgb((a^2-1)b_1)=
\lgb\left(\frac{a^2-1}{a^2}b\right).
\end{equation}
Then $x=\lgb d$, where $d=(a^2-1)b_1=b-b_1\in D$, so $x\in L$.
Moreover,  since $\alpha=\lgb a$, we have
\begin{equation*}
x+2\alpha= \lgb\left(\frac{a^2-1}{a^2}b\right) +2\lgb a = \lgb(a^2-1) + 1,
\end{equation*}
so $x=\lgb \dP -2\alpha+1$ with $\dP =a^2-1\in D$. Thus, $x\in L-2\alpha$
and hence 
$x\in L\cap (L-2\alpha)$.

On the other hand, we have
\begin{equation*}
x+\alpha= \lgb\left(\frac{a^2-1}{a^2}b\right) +\lgb a = 
\lgb\left(\frac{a^2-1}{a}\right)   + 1,
\end{equation*}
and since $(a^2-1)/a$ is not an integer, but falls into the interval 
$1<(a^2-1)/a<b$, it follows that $x\not\in L-\alpha$.
Thus, \eqref{eq:main2b} holds, and the proof of Theorem \ref{thm:main2} is
complete.
\end{proof}

%%%%%%%%%%%%%%%%%%%%%%%%%%%%%%%
% section 4 
%%%%%%%%%%%%%%%%%%%%%%%%%%%%%%%

\section{Extreme values of the complexity function $\pab(n)$}
\label{sec:extreme}

\begin{thm}[Extreme values of the complexity of $\Sab$] 
\label{thm:extreme}
Let $(a,b)$ be an admissible pair. Then we have
\begin{align}
\label{eq:cab-bounds}
\Fl{\frac{b-1}{2}} &\le \cab\le b-1,
\\
\label{eq:pab-bounds}
\Fl{\frac{b-1}{2}}n 
+\Ceil{\frac{b-1}{2}} 
&\le \pab(n)\le (b-1)n.
\end{align}
Moreover, the bounds 
in \eqref{eq:cab-bounds} and \eqref{eq:pab-bounds} are sharp: The upper
bounds are attained for $a=b+1$, while the lower bounds
are attained for $a=2$.
\end{thm}

\begin{proof}
First note that, by \eqref{eq:dab}, we can write 
\[
\pab(n)=\cab n + (b-1)-\cab = (b-1)+ \cab(n-1).
\]
Thus, the maximal and minimal values of $\pab(n)$ are achieved when
$\cab$ is maximal and minimal, respectively, and it therefore suffices 
to prove the bounds \eqref{eq:cab-bounds} for $\cab$.

The upper bound $\cab\le b-1$ in \eqref{eq:cab-bounds}
follows immediately from formula \eqref{eq:cab} of Theorem \ref{thm:main},
and Corollary \ref{cor:integer-a} shows that this bound is attained for
$a=b+1$.

Now consider the lower bound in \eqref{eq:cab-bounds}.
When $a=2$, this bound follows from Corollary \ref{cor:a=2}, which also
shows that the bound is attained in this case.  By the symmetry property
(Corollary \ref{cor:symmetry}), the same conclusion holds for $a=b/2$. 

We now consider the case when $a\not=2$ and $a\not=b/2$.
Without loss of generality we may assume $1<a<b$, so that 
we have $a=r/s$ in the representation \eqref{eq:a=bkrs}. 
Since, by the symmetry property, 
$p_{a,b}(n)=p_{b/a,b}(n)$, we may further restrict the range for $a$ to 
$1<a<\sqrt{b}$. (Note that we do not need to consider the case 
$a=\sqrt{b}$ since then the pair $(a,b)$ is not admissible.)
Thus we have $a=r/s$, were $r$ and $s$ are positive
integers satisfying 
\begin{equation}
\label{eq:rs-cond}
1<\frac{r}{s}<\sqrt{b},\quad (r,s)=1.
\end{equation}
Note that, since $r/s>1$, we necessarily have $r\ge 2$. Moreover, the
case $r=2$ can only occur when $s=1$, but this reduces to the case 
$a=r/s=2$ we considered above. Thus we may assume that $r\ge 3$.
In this case, the bounds \eqref{eq:rs-cond}
imply 
\begin{equation}
\label{eq:rs-cond2}
3\le r< s\sqrt{b}.
\end{equation}

By Theorem \ref{thm:main} we have
\[
\cab =b-1-\Fl{\frac{b-1}{r}} -\Fl{\frac{(b,r)-1}{s}}.
\]
Setting  
\begin{equation}
\label{eq:frs-def}
f(r,s)=
\Fl{\frac{b-1}{r}} +\Fl{\frac{(b,r)-1}{s}},
\end{equation}
the desired bound $\cab\ge \fl{(b-1)/2}$ is seen to be equivalent to 
\begin{equation}
\label{eq:frs-bound}
f(r,s) \le \Ceil{\frac{b-1}{2}}.
\end{equation}
To complete the proof of Theorem \ref{thm:extreme},
it then remains to 
show that the inequality 
\eqref{eq:frs-bound} holds whenever $r,s$ are positive
integers satisfying \eqref{eq:rs-cond} and \eqref{eq:rs-cond2}.

We break the argument into several cases.

\medskip

\textbf{Case I: $r\ge b$.}
By \eqref{eq:rs-cond2}, 
we have in this case $s>r/\sqrt{b}\ge \sqrt{b}$ and hence $s>2$.
Therefore,
\begin{equation}
\label{eq:frs-caseI}
f(r,s)= \Fl{\frac{b-1}{r}} +\Fl{\frac{(b,r)-1}{s}}
\le \Fl{\frac{b-1}{b}}+ 
\Fl{\frac{b-1}{2}}< \Ceil{\frac{b-1}{2}},
\end{equation}
so the inequality \eqref{eq:frs-bound} holds. 

\medskip

\textbf{Case II: $b/3\le r<b$, $b\ge 16$.}
First note that since $r<b$ we have $(b,r)\le b/3$ unless $b$ is even
and $r=b/2$.  In the latter case, we may assume $s\ge 2$ since 
$s=1$ would reduce to the special case $a=r/s=b/2$ 
we considered earlier.  Thus we have
\begin{align*}
\Fl{\frac{(b,r)-1}{s}}
&\le 
\begin{cases} 
\Fl{\dfrac{b}{4}-\dfrac12} &\text{if $b$ is even, $r=b/2$, $s\ge2$,}
\\[2ex]
\Fl{\dfrac{b}{3}-1} &\text{otherwise}
\end{cases}
\\
&\le \frac{b}{3}-\frac{1}{2}
=\frac{b-1}{2}-\frac{b}{6} \le \frac{b-1}{2}-\frac{8}{3},
\end{align*}
where in the last step we used the assumption $b\ge 16$. 
Taking into account the lower bound $r\ge b/3$, we obtain
\begin{align}
\notag
f(r,s) &\le \Fl{\frac{b-1}{b/3}} +\frac{b-1}{2}-\frac{8}{3}
\\
\notag
&=2+ \frac{b-1}{2}-\frac83 < \Ceil{\frac{b-1}{2}}.
\end{align}
Thus \eqref{eq:frs-bound} holds in this case.

\medskip

\textbf{Case III: $3\le r<b/3$, $b\ge 16$.}
We have 
\begin{equation}
\label{eq:fr1a}
f(r,s)\le
\Fl{\frac{b-1}{r}}+(b,r)-1\le
\Fl{\frac{b-1}{r}+r-1} = \Fl{g(r)},
\end{equation}
say, where 
\[
g(r)=\frac{b-1}{r}+r-1.
\]
Since $g'(r)=-(b-1)r^{-2}+1$, the function $g(r)$ is
decreasing in the range $3\le r<\sqrt{b-1}$ and increasing 
in the range $r>\sqrt{b-1}$. Hence, the maximal value of this function 
on the interval $3\le r\le b/3$ 
occurs at one of the endpoints, $r=3$ and $r=b/3$. 
Now, since $b\ge 16$, we have 
\begin{align*}
g(b/3)&=\frac{b-1}{b/3} + \frac{b}{3}-1
\le 3+\frac{b}{3}-1=\frac{b-1}{2}-\frac{b-15}{6}
<\frac{b-1}{2},
\\
g(3)&=\frac{b-1}{3}+2
=\frac{b-1}{2}-\frac{b-1}{6}+2
<\frac{b-1}{2},
\end{align*}
since $b\ge 16$.
Thus we have the bound
\begin{equation}
\label{eq:fr1b}
f(r,s)
\le 
\max(g(b/3),g(3))<\frac{b-1}{2}\le \Ceil{\frac{b-1}{2}},
\end{equation}
which proves \eqref{eq:frs-bound} for Case III.

\medskip

\textbf{Case IV: $3\le r<b$,  $5\le b\le 15$.}
This case can be handled by direct computation of $f(r,s)$ for all pairs
$(r,s)$ of positive integers satisfying \eqref{eq:rs-cond2} and
comparing these values with $\ceil{(b-1)/2}$.
\end{proof}

%%%%%%%%%%%%%%%%%%%%%%%%%%%%%%%
% section 5 
%%%%%%%%%%%%%%%%%%%%%%%%%%%%%%%

\section{Asymptotic and average behavior of the complexity function
$\pab(n)$} \label{sec:asymp}

In this section we consider two natural questions about the complexity
of leading digit sequences $\Sab$:

\begin{itemize}
\item[(1)] Given a sequence $\{a^n\}$, 
how does the complexity of the associated leading digit sequence 
$\Sab$ behave as the base $b$ tends to infinity?
\item[(2)] Given a base $b$, what can we say about the ``average'' complexity 
of the leading digit sequences $\Sab$?
\end{itemize}

%%%%%%%%%%%%%%%%%%%%%%%%%%%%%%%
% asymptotic behavior 
%%%%%%%%%%%%%%%%%%%%%%%%%%%%%%%

We will focus mainly on the case when $a$ is an integer.
Figure \ref{fig:asymptotic} provides numerical data on these
questions.

\begin{figure}[H]
\begin{center}
\includegraphics[width=0.48\textwidth]{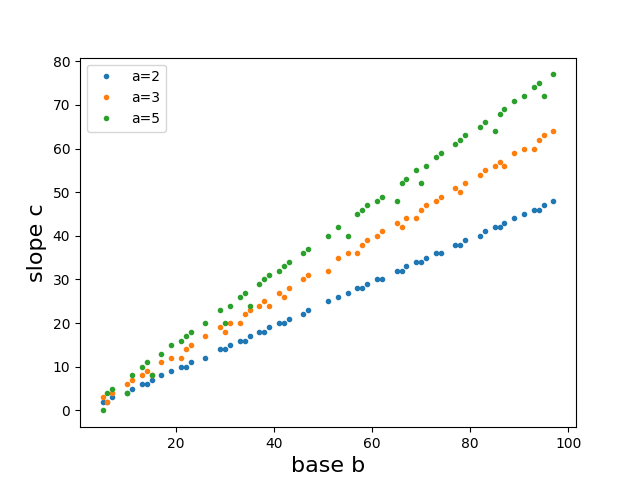}
\hspace{1em}
\includegraphics[width=0.48\textwidth]{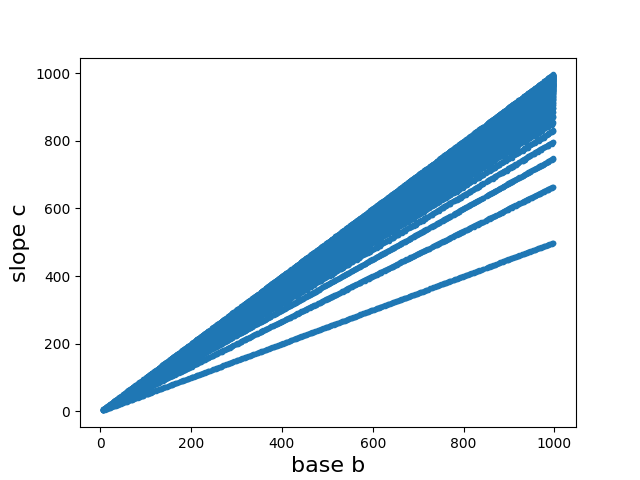}
\end{center}
\caption{The figure on the left shows the slope $c=\cab$
of the complexity function
of the leading digit sequence $\Sab$, as a function of the base
$b$ (restricted to squarefree values),
for each of the values $a=2,3,5$. The figure on the right shows
the set of \emph{all} slopes $\cab$, as $a$ runs through 
integer values $2,3,\dots,b-1$.  The values $\cab$ were computed 
using formula \eqref{eq:cab} of Theorem \ref{thm:main}.
}
\label{fig:asymptotic}
\end{figure}

Figure \ref{fig:asymptotic} suggests that, as $b\to\infty$,
the slope $\cab$ is
asymptotically proportional to $b$, with the proportionality constant
depending on the value $a$.  In the following theorem we show that this
is indeed the case and we determine the proportionality constant involved.

\begin{thm}[Asymptotic behavior of the complexity as $b\to\infty$]
\label{thm:asymptotic}
Let $a$ be a fixed integer $\ge 2$, and suppose 
$b$ tends to infinity through squarefree values.
Then we have
\begin{align}
\label{eq:cab-asymptotic}
\cab&=\left(1-\frac1a\right)b+O(1),   
\end{align}
and, for any fixed integer $n\ge1$, 
\begin{align}
\label{eq:pab-asymptotic}
\pab(n)&= \left(1-\frac1a\right)nb +\frac{b}{a} + O(1).
\end{align}
\end{thm}

\begin{proof}
By part (i) of Corollary \ref{cor:integer-a} we have
\[
\cab=b-\Fl{\frac{b-1}{a}} -(a,b),
\]
provided $b>a$ and $b$ is squarefree. Using the inequalities 
$1\le (a,b)\le a$ and $t-1\le \fl{t}\le t$ it follows that
\begin{align*}
\cab
&\le b-\frac{b-1}{a}+1-(a,b) \le \left(1-\frac1a\right)b +\frac1a,
\\
\cab
&\ge b-\frac{b-1}{a}-(a,b) \ge \left(1-\frac1a\right)b -a,
\end{align*}
which yields the first relation of the theorem, 
\eqref{eq:cab-asymptotic}.

The second relation, \eqref{eq:pab-asymptotic},
follows from the identity
(see \eqref{eq:pab-rational} and \eqref{eq:dab})
\begin{equation}
\pab(n)=\cab n + (b-1-\cab),
\end{equation}
upon substituting the estimate \eqref{eq:cab-asymptotic} for $\cab$.
\end{proof}

We now turn to second question above, concerning 
the average behavior of the complexity function. In the theorem below
we give an asymptotic estimate for 
the average slope of the complexity function $\pab(n)$, 
as $a$ runs through the $b-2$ integers $a=2,3,\dots,b-1$.   
Note that, if $b\ge5$ is squarefree, then for each of these integers 
$(a,b)$ is admissible.

Let
\begin{equation}
\label{eq:cbar}
\cbar =\frac1{b-2}\sum_{a=2}^{b-1}\cab.
\end{equation}
be the average of the slopes $\cab$, taken
over all integers $a$ in the interval $2\le a\le b-1$.

\begin{thm}[Average behavior of the complexity]
\label{thm:average}
With the above notation we have, for any fixed $\epsilon>0$,  
\begin{equation}
\label{eq:cbar-asymptotic}
\cbar = b +O_\epsilon\left(b^{\epsilon}\right),
\end{equation}
where the notation $O_\epsilon(\dots)$ means that the implied constant in the
$O$-term depends on $\epsilon$.
\end{thm}

\begin{proof}
By part (i) of Corollary \ref{cor:integer-a} we have
\begin{align}
\label{eq:cbar-estimate}
\sum_{a=2}^{b-1}
\cab
&=
\sum_{a=2}^{b-1}
\left(b -\Fl{\frac{b-1}{a}} - (b,a)\right)
=b(b-2)
- \sum_{a=2}^{b-1}
\Fl{\frac{b-1}{a}}
- \sum_{a=2}^{b-1}
(b,a).
\end{align}
Let $S_1$ and $S_2$ denote the last two sums.  Then
\begin{equation}
\label{eq:S1}
S_1\le \sum_{a=2}^{b-1}\frac{b-1}{a}=(b-1)(\log (b-1) + O(1)) = 
O\left(b\log
b\right).
\end{equation}
Moreover, 
\begin{align}
\label{eq:S2}
S_2&\le \sum_{a=1}^{b}(a,b)
\le \sum_{d|b}d\, \#\{a\le b-1: (a,b)=d\}
\\
\notag
&\le \sum_{d|b}d\, \#\{a\le b: d | a \}
\le \sum_{d|b}d \frac{b}{d}
\\
&= b \tau(b) = 
O_\epsilon\left(b^{1+\epsilon}\right),
\notag
\end{align}
where $\tau(b)$ denotes the number of divisors of $b$, and in the last step
we have used the estimate (see, e.g.,  \cite[Theorem 315]{hardy-wright})
\[
\tau(b)= O_\epsilon\left(b^{\epsilon}\right),
\]
which holds for any fixed $\epsilon>0$.
Combining \eqref{eq:cbar-estimate}, \eqref{eq:S1}, and \eqref{eq:S2}, 
we get 
\begin{equation}
\cbar =  b + O_\epsilon(\log b) + O_\epsilon\left(b^{\epsilon}\right)
=b+O_\epsilon\left(b^{\epsilon}\right),
\end{equation}
as claimed.
\end{proof}

%%%%%%%%%%%%%%%%%%%%%%%%%%%%%%%
% section 6 
%%%%%%%%%%%%%%%%%%%%%%%%%%%%%%%

\section{The set of complexity functions $\pab(n)$}
\label{sec:values}

A fundamental question in the complexity theory of sequences is which
functions can arise as the complexity function $p_S(n)$ of some sequence
$S$. There exists a large body of results in the literature establishing
necessary or sufficient conditions on a complexity function; see Ferenczi
\cite{ferenczi1999} for a survey. 
In particular, a result of Cassaigne \cite[Th\'eor\`eme 5.3]{cassaigne1997} 
implies that any function of the form $cn+d$, where $c$ and $d$ are
\emph{positive} integers, is the complexity function of some sequence $S$ for all $n\ge 1$.  

By Theorem \ref{thm:main}, if $(a,b)$ is admissible, then 
the complexity function $\pab(n)$ of the
leading digit sequence $\Sab$ is necessarily an affine function, i.e.,  of the form
$cn+d$. In light of the result mentioned above, the theorem therefore does
not give rise to new classes of complexity functions. However, one can ask
which functions $cn+d$ can be obtained as complexity functions of a
sequence of the special form $\Sab$. In this section 
we address this question.  We begin with the following definition.

\begin{defn}[Leading Digit Complexity Function and Good Pairs]
\label{def:good-function}
\mbox{}
\begin{itemize}
\item[(i)]
A function $cn+d$ a called a
\emph{leading digit complexity function} if 
there exists an admissible pair $(a,b)$  
such that $cn+d=\pab(n)$ for all $n$.

\item[(ii)]
A pair $(c,d)$ of integers is called \emph{good} if 
it is the pair of coefficients of a leading digit complexity
function $cn+d$.
\end{itemize}
\end{defn}

We define the sets 
\begin{align}
\label{eq:def-g}
G&=\{(c,d): \text{ $cn+d$ is a leading digit complexity function}\},
\\
\label{eq:def-gc}
G(c)&=\{d: (c,d)\in G\}.
\end{align}
Thus, $G$ is the set of all ``good'' pairs $(c,d)$, 
and $|G(c)|$ is the number of good pairs with first coordinate $c$.

Figure \ref{fig:good-pairs-number}
shows the behavior of $|G(c)|/\sqrt{c}$ as a function of $c$
and of $(\sum_{c\le N}|G(c)|) N^{-3/2}$ as a function of $N$. 
The data suggests the first of these two functions is bounded above and
below by positive constants,
but does not converge to a limit, 
while the second function appears to converge to a limit. \
\begin{figure}[H]
\begin{center}
\includegraphics[width=0.49\textwidth]{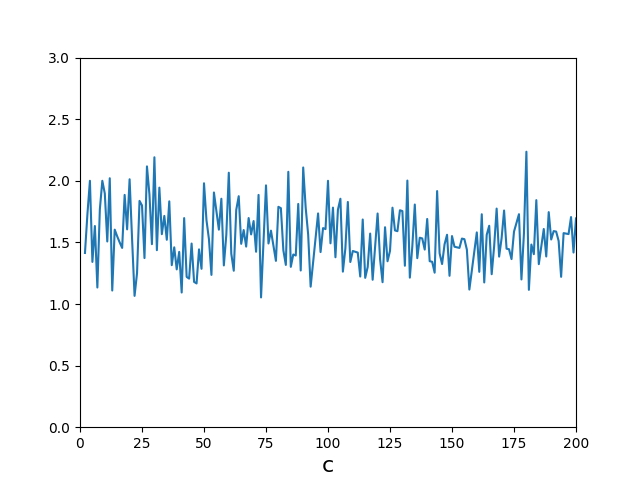}
\hspace{.2em}
\includegraphics[width=0.49\textwidth]{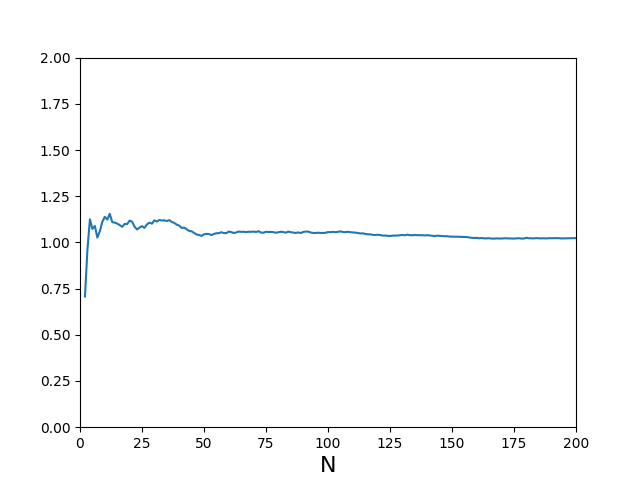}
\end{center}
\caption{Number of good pairs $(c,d)$: 
The figure on the left shows the behavior of the function
$|G(c)|/\sqrt{c}$ as a function of $c$.
The figure on the right shows the behavior of
$(\sum_{c\le N}|G(c)|) N^{-3/2}$ as a function of $N$.} 
\label{fig:good-pairs-number}
\end{figure}

Motivated by such numerical data, we make the following conjecture:

\begin{conj}[Number of Good Pairs]
\label{conj:good-pairs-number}
\mbox{}
\begin{itemize}
\item[(i)] 
There exist 
positive constants $k_1$ and $k_2$  such that
\begin{equation}
\label{eq:good-pairs-bound}
k_1\sqrt{c}\le |G(c)| \le k_2\sqrt{c}
\end{equation}
for all sufficiently large $c$,
but the limit
\begin{equation}
\label{eq:good-pairs-limit}
\lim_{c\to\infty}
\frac{|G(c)|}{\sqrt{c}}
\end{equation}
does not exist.

\item[(ii)] 
There exists a positive constant $k$ such that 
\begin{equation}
\label{eq:good-pairs-asymptotic}
\sum_{c\le N}|G(c)| \sim k N^{3/2}\quad (N\to\infty).
\end{equation}
\end{itemize}
\end{conj}

The numerical data presented in Figure \ref{fig:good-pairs-number}
suggests that we can take $k_1=1$ and $k_2=2.5$ in 
\eqref{eq:good-pairs-bound}.
and $k=1$ in \eqref{eq:good-pairs-asymptotic}

%%%%%%%%%%%%%%%%%%%%%%%%%%%%%%%
% section 7
%%%%%%%%%%%%%%%%%%%%%%%%%%%%%%%

\section{Other complexity measures}
\label{sec:cyclomatic}

In a series of papers in the early 1980s (see \cite{iyengar1983},
\cite{kak1983}, \cite{rajagopal1984}), S. Iyengar, A.K. Rajagopal, S.C. Kak
and others studied the complexity of the sequence of leading digits of
$2^n$ using graph-theoretic complexity measures.  In this section, we
describe this approach, and we determine explicitly 
the complexity of sequences $\Sab$ with respect to a particular graph-theoretic
complexity measure, the so-called \emph{cyclomatic complexity}.

Cyclomatic complexity is a well-known complexity measure for graphs that is
widely  used as a measure for the complexity of computer programs.

\begin{defn}[Cyclomatic Complexity of a Graph (McCabe \cite{mccabe1976},
Berge \cite{berge1973})]
\label{def:cyclomatic-complexity-graphs}
Let $G$ be a finite directed graph. The \emph{cyclomatic complexity} of $G$,
$C_G$, is defined as 
\begin{equation}
\label{eq:def-cyclomatic-complexity-graphs}
C_G= e-n +p,
\end{equation}
where $e$ is the number of (directed) edges, $n$ the number of vertices,
and $p$ the number of connected components of the graph $G$.
\end{defn}

In order to apply this concept to the complexity of a \emph{sequence},
one has to associate a graph to the sequence.  Iyengar et al. \cite{iyengar1983} 
suggest several ways to do so.
The simplest, and most natural, approach is  
to consider the \emph{transition graph}, $G_S$, of the sequence $S$,
defined as the directed graph whose vertices are the symbols in $S$, and
which contains an edge from $a$ to $b$ if and only if $a$ and $b$ occur in
consecutive positions in the sequence $S$.  We thus make the following
definition.

\begin{defn}[Cyclomatic Complexity of a Sequence $S$]
\label{def:cyclomatic-complexity-sequence}
Let $S$ be an infinite sequence over a finite set of symbols, and let $G$
be its transition graph.  
The \emph{cyclomatic complexity}, $C_S$,  of the sequence $S$
is defined as the cyclomatic complexity of the transition graph $G$.
\end{defn}

%%%%%%%%%%%%%%%%%%%%%%%%%%%%%%%
% new version
%%%%%%%%%%%%%%%%%%%%%%%%%%%%%%%

Under a mild additional assumption on $S$ (which amounts to a weak type 
of \emph{recurrence}), we have the following connection between the
cyclomatic complexity, $C_S$, of a sequence
and its block complexity, $p_S(n)$.

\begin{lem}[Cyclomatic Complexity and Block Complexity]
\label{lem:cyclomatic-block-complexity}
Let $S$ be an infinite sequence over a finite set of symbols and assume that 
each symbol occurring in $S$ occurs infinitely often. Then we have
\begin{equation}
\label{eq:cyclomatic-block-complexity}
C_S=p_S(2)-p_S(1) + 1.
\end{equation}
\end{lem}

\begin{proof}
Let $G$ be the transition graph of $S$, and let $n$, $e$, and $p$
denote, respectively, the number of vertices, directed edges, and connected
components of $G$.

The number of vertices in $G$
is the number of symbols in the sequence,
which in turn is equal to the number 
of distinct blocks of length $1$ in the sequence,
i.e., the quantity $p_S(1)$. Thus, we have
\begin{equation}
\label{eq:n=pS1}
n=p_S(1).
\end{equation}

Next, observe that there is a one-to-one correspondence 
between edges in $G$ and
pairs $(d_1,d_2)$ of consecutive terms in the sequence $\Sab$. Indeed, by
the definition of the transition graph of a sequence, there is an edge from
$d_1$ to $d_2$ if and only if $d_1$ and $d_2$ occur as consecutive terms in the
sequence.  Thus, the number of edges in $G$ is equal to the number of
distinct pairs $(d_1,d_2)$ of consecutive terms in the sequence. But the
latter number is the number of distinct blocks of length $2$ in the
sequence, so we have
\begin{equation}
\label{eq:e=pS2}
e=p_S(2).
\end{equation}

Finally, we will show that the graph $G$ has only one connected
component, i.e., that
\begin{equation}
\label{eq:p=1}
p=1.
\end{equation}
Indeed, by our assumption that each term $d$ that occurs in
$S$ occurs there infinitely often, it follows that, given any two such 
terms, $d_1$ and $d_2$, the sequence must
contain a string of consecutive terms beginning with $d_1$ and ending with
$d_2$. By the definition of the transition graph $G$, this means that
there is a path from $d_1$ to $d_2$. Since $d_1$ and $d_2$ were arbitrary
terms (i.e., arbitrary vertices in $G$), it follows that the graph can have only
one connected component, proving \eqref{eq:p=1}.

Substituting \eqref{eq:n=pS1}, \eqref{eq:e=pS2}, and \eqref{eq:p=1} into
\eqref{eq:def-cyclomatic-complexity-graphs},
we obtain the desired relation 
\eqref{eq:cyclomatic-block-complexity}.
\end{proof}

We now focus on the case of leading digit sequences of the form $\Sab$, 
and we denote the cyclomatic complexity of such a sequence by $\Cab$, i.e., 
we set $\Cab=C_S$, where $S=\Sab$.
Combining Lemma \ref{lem:cyclomatic-block-complexity} with Theorem
\ref{thm:main}, we can determine $\Cab$ explicitly for any for any admissible
pair $(a,b)$:

\begin{cor} [Cyclomatic Complexity of $\Sab$]
\label{cor:cyclomatic}
Let $(a,b)$ be an admissible pair, and let
$\Sab$ be the sequence of leading digits of $a^n$ in base $b$.
Then the cyclomatic complexity of $\Sab$ is given by 
\begin{equation}
\label{eq:cyclomatic-complexity}
\Cab= b-\Fl{\frac{b-1}{r}} -\Fl{\frac{(b,r)-1}{s}},
\end{equation}
where $r$ and $s$ are defined as in Theorem \ref{thm:main}, i.e., as the
unique integers satisfying  
\begin{equation}
\label{eq:a=bkrs2}
a=\frac{r}{s}\,b^k,
\quad k\in\ZZ,\quad r,s\in\NN, \quad (r,s)=1,
\quad 1<\frac{r}{s}<b.
\end{equation}
\end{cor}

\begin{proof}
Let $\Sab$ be as in the statement.  It is easy to see (cf. the argument
following \eqref{eq:p(k)-Lk3}) that each digit $d\in\{1,2,\dots,b-1\}$
occurs infinitely often in $\Sab$. Thus $\Sab$ satisfies the hypothesis of 
Lemma \ref{lem:cyclomatic-block-complexity}, and hence has cyclomatic
complexity given by $\Cab = \pab(2)-\pab(1)+1$.  
Substituting the formulas \eqref{eq:pab-rational} and \eqref{eq:cab} 
from Theorem \ref{thm:main}, it follows that 
\begin{equation*}
\Cab=\cab+1 = b-\Fl{\frac{b-1}{r}} -\Fl{\frac{(b,r)-1}{s}},
\end{equation*}
which is the desired formula \eqref{eq:cyclomatic-complexity}.
\end{proof}

%%%%%%%%%%%%%%%%%%%%%%%%%%%%%%%
% section 8
%%%%%%%%%%%%%%%%%%%%%%%%%%%%%%%

\section{Concluding remarks}
\label{sec:concluding}

In this section we discuss some related concepts and open questions
suggested by our results.

\paragraph{Rauzy graphs.}
In Section \ref{sec:cyclomatic} we defined the cyclomatic complexity of a
sequence $S$ as the (graph-theoretic) cyclomatic complexity of the
transition graph $G_S$ associated with this sequence. 
This transition graph is a particular case of a family of graphs
associated with the sequence $S$, known as \emph{Rauzy graphs}, and defined
as follows:
Given a sequence $S$, the Rauzy graph of
level $n$, $\Gamma_n(S)$,  
is the directed graph whose vertices are the distinct 
``blocks'' of length $n$ occurring in $S$, and in which two blocks 
of length $n$ are connected by a directed edge if and only if the second
block ``continues'' the first block in the sense that it overlaps with the
first block in its first $n-1$ positions; see Arnoux and Rauzy
\cite{arnoux-rauzy} and also Section 2.1 of \cite{alessandri-berthe}.

The Rauzy graph $\Gamma_1(S)$ is the transition graph $G_S$ we
have used to define the cyclomatic complexity of a sequence $S$.
We remark that for sequences $\Sab$ the cyclomatic complexity of
the Rauzy graph $\Gamma_n(\Sab)$ is independent of $n$: Indeed, the graph $\Gamma_n(\Sab)$
has $\pab(n)$ vertices, $\pab(n+1)$ edges, and one connected component, so
its cyclomatic complexity is $\pab(n+1)-\pab(n)+1=\cab+1$, where $\cab$ is
the ``slope'' of $\pab(n)$, given by \eqref{eq:cab}. 

\paragraph{Another graph-theoretic complexity measure for sequences.} 
In their paper \cite{iyengar1983}, Iyengar et al. proposed an interesting 
graph-theoretic complexity measure for the leading digit sequence of $\{2^n\}$ 
that is different from the one we considered in the previous section.  It is based
on the remarkable fact, established in \cite{iyengar1983}, that the sequence of
leading digits of $2^n$ can be completely decomposed into the five blocks $a=1248$,
$b=1249$, $c=125$, $d=136$, and $e=137$.  Rewriting the sequence as a
sequence in the symbols $a,b,c,d,e$, one can then consider the associated transition
graph between these symbols. This graph is different from the simple transition
graph, and also from the general Rauzy graphs $\Gamma_n(S)$ considered above. Yet, as 
Iyengar et al. have shown, when $S$ is the leading digit sequence of $\{2^n\}$, all 
of these graphs have the same cyclomatic complexity, namely $5$.

Iyengar et al. focused mainly on the leading digit sequence of $\{2^n\}$. It would
be interesting to see if their approach can be extended to the more general
leading digit sequences $\Sab$ we have considered in the present paper.

\paragraph{Complexity functions of other ``natural'' arithmetic sequences.}
A key motivation for the present work was to completely determine the
complexity function for a natural class of sequences of arithmetic
interest, namely the sequences $\Sab$ of leading digits of $a^n$ in base $b$.
Another class of arithmetic sequences whose complexity has been analyzed in
a similarly systematic manner are sequences obtained as expansions with
respect to an irrational base $\beta>1$; see, e.g.,  Frougny et al. \cite{frougny2004}
and Klouda and Pelantov\'a \cite{klouda2009}.

As a natural extension of our results on the complexity of the sequences
$\Sab$, one can try to determine the complexity of more general leading
digit sequences, such as the leading digits of $\{2^{n^2}\}$, $\{n!\}$, and
$\{n^n\}$.  Recent work \cite{local-benford} 
on the local distribution of sequences 
of this type suggests that these sequences have relatively low
complexity, possibly of polynomial rate of growth. 
On the other hand, in \cite{local-benford} it was also shown that 
for ``almost all'' doubly exponential sequences $\{a^{\theta^n}\}$
the associated leading digit sequences behave locally like independent 
Benford-distributed random variables and thus 
have maximal complexity, i.e., satisfy $p(n)=9^n$ (in the case of base $10$).
Interestingly, recent numerical investigations \cite{mersenne-benford}
suggest  that the same holds for the much slower growing sequence of
Mersenne numbers $\{2^{p_n}-1\}$, where $p_n$ denotes the $n$-th prime
number.  Proving results of this type, however, seems to be well out of reach. 

%%%%%%%%%%%%%%%%%%%%%%%%%%%%%%%
% acknowledgements
%%%%%%%%%%%%%%%%%%%%%%%%%%%%%%%

\acknowledgements

We are grateful to the referees for their thorough reading of the paper
and many helpful comments and suggestions, which, in particular, led to a
strengthening of the statement of Theorem 5.1.

\bibliographystyle{alpha}
\bibliography{complexity-references}

\end{document}